\documentclass[12pt,a4paper]{amsart}
\usepackage[utf8]{inputenc}
\usepackage[T1]{fontenc}
\usepackage{lmodern}
\usepackage[english]{babel}
\usepackage{amsmath,mathtools,amsthm,amssymb,amsfonts}
\usepackage{color}
\usepackage{geometry}
\usepackage{mathrsfs}
\usepackage{graphicx}
\usepackage{enumitem}
\usepackage{tikz} 
\usepackage{pgfplots}
\usepackage{esvect}
\usepackage{pdfpages}
\usetikzlibrary{shapes,arrows}
\usetikzlibrary{calc}
\usetikzlibrary{shapes.geometric}
\usetikzlibrary{shapes.arrows}
\usetikzlibrary{arrows.meta}
\usetikzlibrary{decorations.markings}
\usetikzlibrary{fit}
\usetikzlibrary{patterns}
\usetikzlibrary{hobby}
\usepgfplotslibrary{patchplots}
\pgfplotsset{compat=1.11}
\usepackage{hyperref}
\usepackage{nicefrac}
\usepackage{cancel}
\usepackage{parskip} 
\usepackage{enumitem}
\usepackage{float}

\theoremstyle{plain}
\newtheorem{theorem}{Theorem}[section]
\newtheorem{corollary}[theorem]{Corollary}

\theoremstyle{definition}
\newtheorem{definition}[theorem]{Definition}
\newtheorem{remark}[theorem]{Remark}

\newtheorem*{theorem*}{Theorem}
\newtheorem*{definition*}{Definition}
\newtheorem*{corollary*}{Corollary}
\newtheorem*{remark*}{Remark}
\newtheorem*{thm*}{Theorem}
\newtheorem*{conjecture*}{Conjecture}

\newcommand{\definedas}{\mathrel{\raise.095ex\hbox{\rm :}\mkern-5.2mu=}}
\newcommand{\asdefined}{\mathrel{=\mkern-5.2mu\raise.095ex\hbox{\rm :}}}

\newcommand{\tr}{\operatorname{tr}}

\newcounter{flabelcounter}
\allowdisplaybreaks

\title[]{Curvature Inequalities and Rigidity for Constant Mean Curvature and Spacetime Constant Mean Curvature Surfaces}
\author[]{Alejandro Pe\~nuela Diaz}
\address{University of Potsdam, 14476 Potsdam, Germany}
\email{alejandro.penuela.diaz@uni-potsdam.de}

\begin{document}
\begin{abstract}
We establish curvature inequalities and rigidity results for constant mean curvature (CMC) surfaces in Riemannian manifolds and spacetime constant mean curvature (STCMC) surfaces in Lorentzian manifolds. In three-dimensional Riemannian manifolds with nonnegative scalar curvature, we show that the Christodoulou–Yau inequality $H^2\leq 16\pi/|\Sigma|$ ($H$ denoting mean curvature and $|\Sigma|$ area) for stable CMC spheres also holds for CMC surfaces of arbitrary topology under a new stability condition controlling only the constant mode of the second variation of area. Under an extrinsic curvature sign condition, equality forces the enclosed region to be Euclidean without additional intrinsic assumptions. We extend these results to higher dimensions and to the hyperbolic and spherical settings.

In the Lorentzian setting, we develop a variational stability theory for STCMC surfaces based on the second variation of area in the mean curvature vector direction. We introduce two independent stability notions: variational stability, which recovers classical volume-preserving CMC stability for surfaces in totally geodesic spacelike hypersurfaces, and constant-mode stability. Under the dominant energy condition, either yields the sharp inequality $|\vec H|^2\leq 16\pi/|\Sigma|$. In the equality case, under suitable geometric assumptions, the maximal globally hyperbolic development of the enclosed spacelike region is a Minkowski causal diamond.  These results imply nonnegativity and rigidity of the Hawking quasi-local energy for stable STCMC surfaces. Finally, we analyze known spacelike and null STCMC foliations, showing that asymptotic leaves are stable under positive mass conditions, whereas the local matter density and trace-free extrinsic curvature govern the instability of local foliations.
\end{abstract}

\maketitle

\section{Introduction and Results}
The study of constant mean curvature (CMC) surfaces as critical points of the area functional under a volume constraint is a central theme in differential geometry. Beyond their role as isoperimetric surfaces, stable CMC surfaces satisfy fundamental geometric inequalities that link the curvature of the ambient manifold to intrinsic properties of the surface. A landmark result in this direction was established by Christodoulou and Yau \cite{Chriyau}, who proved that if $(M,g)$ is a three-dimensional Riemannian manifold with nonnegative scalar curvature and $\Sigma \subset M$ is a stable CMC sphere, then
\begin{equation}\label{ineqint}
H^2 \leq \frac{16\pi}{|\Sigma|},
\end{equation}
where $H$ denotes the mean curvature of $\Sigma$ and $|\Sigma|$ its area. Equality is attained by round spheres in Euclidean space. This naturally leads to the corresponding rigidity question: if equality holds, must the enclosed region be isometric to a Euclidean ball?

Rigidity phenomena under scalar curvature lower bounds have been extensively studied in Riemannian geometry. In many cases, the theory of stable minimal and constant mean curvature hypersurfaces plays a central role, as the stability inequality relates ambient scalar curvature to intrinsic geometric quantities. For a general overview of such rigidity results and their geometric context, see the survey of Brendle \cite{brendlerigidity}.

Under additional intrinsic assumptions, such as even
symmetry or sufficient closeness to roundness,
it has been shown that equality in \eqref{ineqint}
forces $\Sigma$ to be isometric to a round sphere
and the enclosed region to be isometric to a
Euclidean ball
\cite{shi2019uniqueness,sun2017rigidity}.  Note that  inequality \eqref{ineqint} was originally motivated by general relativity, as it ensures the nonnegativity of the Hawking quasi-local energy, while the equality case provides a rigidity statement for the Hawking quasi-local energy within totally geodesic hypersurfaces.

The first goal of this paper is to refine and extend
these results. For CMC surfaces in Riemannian manifolds
with scalar curvature bounds
$\mathrm{Sc}^M\geq 2\Lambda$, $\Lambda\in\mathbb R$,
we establish the unified sharp curvature inequality $
H^2\leq\frac{16\pi}{|\Sigma|}-\frac43\Lambda$,
assuming either classical variational stability or a
new condition involving only the constant mode of the
second variation of area. In the equality case,
we obtain Euclidean, hyperbolic and spherical rigidity
under suitable curvature assumptions, without requiring
intrinsic symmetry or near-roundness.
We also obtain higher-dimensional analogues.

We then turn to the Lorentzian setting, which forms the central part of
this paper. In a general four-dimensional Lorentzian spacetime $(M,g)$,
the natural codimension-two analogue of the CMC condition is the
\emph{spacetime constant mean curvature} (STCMC) condition, requiring
the squared Lorentzian norm of the mean curvature vector $\vec H$ to be
constant:
\[
|\vec H|^2 = \mathrm{const}.
\]
STCMC surfaces have become an
important tool in mathematical relativity, notably
through the canonical foliation near spatial
infinity constructed by Cederbaum and Sakovich
\cite{STCMC}, which provides a geometric
characterization of the center of mass of isolated
systems. Other important examples include local
concentrating foliations \cite{main1local} and the
asymptotic STCMC foliations of Schwarzschildean
lightcones constructed by Kr\"oncke and Wolff
\cite{kroncke2024foliations}. Despite their growing role, however, a
variational stability theory for STCMC surfaces, analogous to
classical CMC stability in Riemannian geometry, has been missing. In particular, the stability operators arising from the linearization of the spacetime mean curvature within spacelike or null hypersurfaces are generally non-self-adjoint \cite{STCMC,kroncke2024foliations}, which limits the direct applicability of classical variational and spectral methods.

Our main goal is to introduce a canonical variational
stability theory for STCMC surfaces
and establish sharp curvature inequalities,
spacetime rigidity results, and applications to
known STCMC foliations. Since STCMC surfaces are precisely the critical points of the area functional under mean-zero variations along
$\alpha\vec H$, we study the second variation operator $\mathcal L$
associated with variations in the distinguished mean curvature
direction $\vec H$. A key feature of this approach is that $\mathcal L$ is self-adjoint and elliptic whenever the mean curvature vector is spacelike. Decomposing into mean-zero variations and the constant mode leads to two complementary stability notions: \emph{variational
stability} and \emph{constant-mode stability}. In a totally geodesic hypersurface, the operator reduces identically to $\mathcal L = 2H^2 L_{\mathrm{Jac}} + H^4$,  where $L_{\mathrm{Jac}}$ is the classical Jacobi operator, so variational stability agrees exactly with classical volume-preserving CMC stability.

The dominant energy condition plays a role analogous
to that of the scalar curvature lower bound in the
Riemannian setting. Under this condition, either
stability criterion yields the sharp spacetime
analogue of the Christodoulou--Yau inequality,
\begin{equation}\label{inequint2}
|\vec H|^2 \le \frac{16\pi}{|\Sigma|}.
\end{equation}
Moreover, we establish rigidity in the
equality case: under an additional curvature
sign condition, or suitable symmetry or
near-roundness assumptions, $\Sigma$ is
intrinsically round. If $\Sigma$ bounds a compact mean-convex spacelike region, this region
embeds isometrically into Minkowski spacetime, and its maximal globally hyperbolic development is isometric to a standard causal diamond in Minkowski spacetime. The proof of this rigidity statement builds on a theorem of Liu and Yau \cite{liu2003positivity,liu2006positivity}, together with an analysis of the  development of the enclosed region.

A central obstacle in codimension-two geometry is that
$|\vec H|^2=\mathrm{const}$ fixes only the length of $\vec H$, leaving
its direction in the normal bundle unconstrained. Classical codimension-two
rigidity results, such as those of Chen and Yau
\cite{chen1973surface,yau1974submanifolds}, rely
on the stronger \emph{parallel mean curvature}
(PMC) condition $\nabla^\perp\vec H=0$. Without
this assumption, the classical rigidity mechanism
does not apply in general. In our setting, however, the dominant energy condition
together with the stability assumptions forces the
mean curvature vector to be parallel in the equality
case. Thus, parallel mean curvature emerges as a
consequence of equality rather than being imposed
as an a priori hypothesis.

These results have a natural interpretation
in terms of the Hawking quasi-local energy
\begin{equation}\label{hawkingen}
\mathcal E_H(\Sigma)
=
\sqrt{\frac{|\Sigma|}{16\pi}}
\left(
1-\frac{1}{16\pi}
\int_\Sigma|\vec H|^2\,d\mu
\right).
\end{equation}
For STCMC surfaces, inequality
\eqref{inequint2} is equivalent to the
nonnegativity of $\mathcal E_H(\Sigma)$,
while equality yields rigidity when the
Hawking energy vanishes.

Finally, we investigate three established
families of STCMC foliations using the
stability notions introduced here. We show that the stability of their leaves is closely tied to physical and geometric parameters of the spacetime: asymptotic leaves are stable under positive mass conditions, whereas the local matter density and the shear of the initial data govern the instability of local foliations.

These results connect the positivity and rigidity properties of the Hawking energy with the geometric structure of STCMC foliations, demonstrating that the stability theory developed here is realized by natural STCMC surfaces, captures physically meaningful features of their geometry, and places the Lorentzian theory in direct analogy with the classical Riemannian theory of stable CMC surfaces.

\subsection{Main results}
We first present a unified curvature inequality and
rigidity theorem for CMC surfaces, covering the
Euclidean, hyperbolic, and spherical reference
geometries.

For a CMC surface $\Sigma$ in a three-dimensional
Riemannian manifold, we say that $\Sigma$ is
$\Lambda$-constant-mode stable if the second
variation of area in the unit normal direction
$\nu$ satisfies
\[
\delta_\nu^2|\Sigma|
=
\int_\Sigma H^2-|B|^2-\mathrm{Ric}^M(\nu,\nu)
 \, d\mu
\geq
\left(\frac12 H^2-\frac23\Lambda\right)|\Sigma|.
\]
The corresponding definition in arbitrary
dimension is given in Definition~\ref{def:cm-riem}.

\medskip

\noindent\textbf{Theorem \ref{thm:riem-3d-unified}.}
\emph{Let $(M,g)$ be a three-dimensional Riemannian
manifold with scalar curvature
$\mathrm{Sc}^M\geq 2\Lambda$ for a constant
$\Lambda\in\mathbb R$.
Let $\Sigma\subset M$ be a closed, connected surface
of constant mean curvature $H$ with respect to a
unit normal $\nu$, satisfying either}
\begin{enumerate}
\item[(i)] \emph{$\Sigma$ is $\Lambda$-constant-mode
stable, or}

\item[(ii)] \emph{$\Sigma$ is topologically a sphere
and is variationally stable.}
\end{enumerate}
\emph{Then}
\begin{equation}\label{inequint1}
H^2\leq\frac{16\pi}{|\Sigma|}-\frac43\Lambda.
\end{equation}
\emph{Suppose additionally that $\Sigma$ is the
mean-convex boundary of a relatively compact
domain $\Omega\subset M$ and that equality holds
in \eqref{inequint1}:}

 \emph{$\bullet$ If $\Lambda \le 0$, $\Omega$ is orientable, and
$\mathrm{Ric}^M(\nu,\nu) - \frac{2}{3}\Lambda$ does not change sign
along $\Sigma$, then $\Sigma$ is isometric to a round sphere and $\Omega$
is isometric to a geodesic ball in the simply connected $3$-dimensional
space form of constant sectional curvature $\Lambda/3$ (Euclidean space
$\mathbb{R}^3$ when $\Lambda=0$, and hyperbolic space
$\mathbb{H}^3_{\Lambda/3}$ when $\Lambda<0$).}

\emph{$\bullet$ If $\Lambda > 0$ and $\mathrm{Ric}^M \ge \frac{2}{3}\Lambda\,g$ on
$\Omega$, then $\Sigma$ is 
isometric to a round sphere, and $\Omega$ is isometric to a geodesic ball
in the round sphere $\mathbb{S}^3(R)$ of radius $R = \sqrt{3/\Lambda}$.
In the minimal case $H=0$ (equivalently $|\Sigma| = 12\pi/\Lambda$), this
geodesic ball is the hemisphere.}

Under constant-mode stability, the curvature
inequality extends to arbitrary dimensions,
and corresponding Euclidean and spherical
rigidity results are established in
Theorem~\ref{thm:riem-hd-unified}.

\subsection*{Results for STCMC surfaces} Let $(M,g)$ be a $4$-dimensional time-oriented Lorentzian manifold.  Motivated by the variational characterization
of STCMC surfaces, we introduce two notions
of stability based on the second variation
of area in the mean curvature vector
direction.

When $|\vec H|^2\geq 0$, we say that a closed
STCMC surface $\Sigma$ is \emph{variationally stable} if
\[
\delta^2_{\alpha\vec H}|\Sigma|
\geq |\vec H|^4\int_\Sigma\alpha^2\,d\mu
\]
for every $\alpha\in C^\infty(\Sigma)$ satisfying
$\int_\Sigma\alpha\,d\mu=0$, and
\emph{constant-mode stable} if
\[
\delta^2_{\vec H}|\Sigma|\geq 0.
\]
We refer to Definition~\ref{def:stcmc-stability}
and the subsequent discussion for the geometric
and operator-theoretic properties of these
stability notions.

Both stability conditions are saturated by
round spheres in  Minkowski spacetime and are in general independent.
Moreover, for CMC surfaces with $H\neq0$
in totally geodesic spacelike hypersurfaces,
variational stability agrees with classical
volume-preserving CMC stability
(Remark~\ref{rem:CMC-reduction}).

We establish a Lorentzian analogue of the above
curvature inequality and rigidity results under
the dominant energy condition. For a spacelike two-surface $\Sigma$, we denote by $\ell$ and $k$ a pair of future-directed null normals normalized by $g(\ell,k)=-2$.

\noindent\textbf{Theorem \ref{thm:rigidity1}.}
\emph{Let $(M,g)$ be a four-dimensional
time-oriented Lorentzian manifold satisfying
the dominant energy condition.
Let $\Sigma$ be a closed, connected spacelike
STCMC surface with $|\vec H|^2\geq0$, satisfying
either}
\begin{enumerate}
\item[(i)] \emph{$\Sigma$ is constant-mode stable,
or}
\item[(ii)] \emph{$\Sigma$ is topologically a
sphere and is variationally stable.}
\end{enumerate}
\emph{Then}
\begin{equation}\label{inequint4}
|\vec H|^2\leq\frac{16\pi}{|\Sigma|}.
\end{equation}
\emph{If equality holds in \eqref{inequint4}
and $\mathrm{Rm}^M(k,\ell,\ell,k)$ does not
change sign along $\Sigma$, then $\Sigma$ is
isometric to a round sphere, and any compact
mean-convex spacelike hypersurface
$\Omega\subset M$ with boundary
$\partial\Omega=\Sigma$ embeds isometrically
as a spacelike hypersurface into Minkowski
spacetime. Furthermore, the maximal globally
hyperbolic development of the induced initial
data on $\Omega$ is isometric to a standard
causal diamond in Minkowski spacetime.}

For variationally stable STCMC spheres, an alternative rigidity criterion replacing the sign condition on
$\mathrm{Rm}^M(k,\ell,\ell,k)$ by even symmetry or intrinsic
near-roundness is proved in Theorem~\ref{thm:rigidity2}.

We also analyze the stability of three geometrically distinguished families of STCMC foliations:

$\bullet$ \emph{Asymptotically Euclidean initial data sets}
(Theorem~\ref{thm:stability-foliation}): the leaves of the canonical
STCMC foliation near spatial infinity \cite{STCMC} are strictly
constant-mode and strictly variationally stable when $E_{\mathrm{ADM}} > 0$
and strictly unstable when $E_{\mathrm{ADM}} < 0$, provided
$|\mathrm{Ein}(\nu_\sigma,\nu_\sigma)| = O(r^{-2})$ along the foliation.

$\bullet$ \emph{Local concentrating STCMC foliations}
(Theorem~\ref{localfoli}): for small STCMC spheres concentrating at a
point $p$ \cite{main1local}, the leaves are strictly variationally
unstable whenever the trace-free extrinsic curvature satisfies
$\mathring K(p) \neq 0$ (independently of the local energy density
$\mu(p)$), and strictly constant-mode unstable whenever $\mu(p) \ge 0$
except possibly when $\mu(p) = 0$ and $\mathring K(p) = 0$. In particular, variational stability and constant-mode stability are independent.

$\bullet$ \emph{Asymptotically Schwarzschildean lightcones}
(Theorem~\ref{nullsta}): the leaves of the asymptotic null STCMC
foliation \cite{kroncke2024foliations} are strictly constant-mode and
strictly variationally stable when the mass is positive.

The remainder of the paper is organized as follows.
Section~\ref{sectionCMC} develops the Riemannian theory.
In Section~\ref{secsetup}, we derive the second variation
formula underlying STCMC stability, and in
Section~\ref{sectionSTCMC}, we establish the
associated curvature inequalities and rigidity results.
Finally, Section~\ref{secfoli} discusses model
examples, relates the stability operator to the geometry
of initial data sets, and analyzes the stability of
known STCMC foliations. The auxiliary rigidity theorems used
throughout the paper are collected in the appendix.


\section{Constant mean curvature surfaces in Riemannian manifolds with scalar curvature bounds}\label{sectionCMC}

We  recall important properties of stable constant mean curvature (CMC) surfaces in a Riemannian manifold with scalar
curvature bounded below. In the context of Lorentzian geometry, such a manifold  can also be viewed as a totally geodesic hypersurface
in an ambient Lorentzian manifold satisfying the dominant energy condition. These results also provide important motivation for the STCMC theory developed here.

A CMC surface is a critical point of the area functional
under volume-preserving variations, that is,
under variations whose normal speed $\alpha$ satisfies
\begin{equation}\label{zeromean}
    \int_\Sigma \alpha\, d\mu = 0.
\end{equation}
We denote by $\delta^2_{\alpha\nu}|\Sigma|$ the unconstrained
second variation of area along the normal-geodesic variation
with initial velocity $\alpha\nu$, where $\nu$ is a unit normal
to $\Sigma$. It is given by
\[
\delta^2_{\alpha\nu}|\Sigma|
=
\int_\Sigma
|\nabla^h\alpha|^2
-
\left(
|B|^2+\mathrm{Ric}^M(\nu,\nu)-H^2
\right)\alpha^2
\,d\mu.
\]
In this notation, a CMC surface is said to be variationally
stable if
\begin{equation}\label{CMCstability}
\delta^2_{\alpha\nu}|\Sigma|
\geq
\int_\Sigma H^2\alpha^2\,d\mu
\end{equation}
for every smooth function $\alpha$ satisfying
\eqref{zeromean}.

Stable CMC surfaces have several remarkable geometric properties. In particular, Christodoulou and Yau proved that, in manifolds with nonnegative scalar curvature, such surfaces satisfy a sharp curvature inequality.
\begin{theorem}[Christodoulou-Yau \cite{Chriyau}]\label{posiyau}
Let $(M,g)$ be a $3$-dimensional Riemannian manifold
with nonnegative scalar curvature.
If $\Sigma\subset M$ is a stable constant mean curvature surface of spherical topology,
then
\[
 H^2\le \frac{16\pi}{|\Sigma|}.
\]
Here, stability is understood in the usual volume-preserving
sense recalled above.
\end{theorem}
 In the time-symmetric setting of an initial data set in general relativity, the Hawking energy of a surface $\Sigma$ reduces to
\[
\mathcal{E}_H(\Sigma)
=
\sqrt{\frac{|\Sigma|}{16\pi}}
\left(
1 - \frac{1}{16\pi}\int_\Sigma H^2\, d\mu
\right).
\]
Thus, the Christodoulou-Yau inequality implies that the Hawking energy is nonnegative for stable CMC surfaces. Related
positivity results for the Hawking energy of CMC surfaces under hypotheses different from stability were obtained in \cite{Miao}.

The inequality in Theorem \ref{posiyau} identifies Euclidean space as the model geometry in the regime of nonnegative scalar curvature. Equality in the Hawking energy bound is achieved by round spheres in $\mathbb{R}^3$. It is therefore natural to ask under what additional conditions equality forces the enclosed region to be isometric to a Euclidean ball. We begin by recalling the following rigidity result.
\begin{theorem}[{\cite[Theorem 2, Theorem 1]{shi2019uniqueness, sun2017rigidity}}]\label{cmcsun}
  Let \((M, g)\) be a $3$-dimensional Riemannian manifold with nonnegative scalar curvature, and let \( \Omega \subset M \) be a relatively  compact domain with smooth boundary \( \Sigma = \partial \Omega \). Assume $\Sigma$ is a stable constant mean curvature sphere satisfying $$H^2= \frac{16\pi}{|\Sigma|} .$$ If  either 
  \begin{enumerate}
\item[(i)]$\Sigma$ has  even symmetry, i.e. there exists an isometry $\rho :\Sigma \to \Sigma $ with $\rho^2 =id$ and $\rho(x)\neq x$ for $x \in \Sigma$, or
\item[(ii)]its Gauss curvature  \( K_{\Sigma} \) is \( \mathcal{C}^0 \)-close to \( \frac{4\pi}{|\Sigma|} \), i.e.  \(
|K_{\Sigma} -\frac{4\pi}{|\Sigma|}|_{\mathcal{C}^0} < \delta_0
\) for some \( \delta_0 \ll 1 \).
\end{enumerate}  
 Then \( \Omega \) is isometric to a Euclidean ball in \( \mathbb{R}^3 \). In particular, \( \Sigma \) is isometric to the round sphere in \( \mathbb{R}^3 \).
\end{theorem}
Analogous rigidity results under the same intrinsic hypotheses (even
symmetry or $\mathcal{C}^0$-closeness to constant Gauss curvature) were
obtained in the hyperbolic reference setting ($\mathrm{Sc}^M \ge 2\Lambda$,
$\Lambda < 0$) in \cite[Theorem~3, Theorem~2]{shi2019uniqueness,sun2017rigidity},
and in the spherical setting ($\mathrm{Sc}^M \ge 6$, with the additional
Ricci lower bound $\mathrm{Ric}^M \ge 2g$ for the enclosed region) by
Melo \cite{melo2024hawking}.

In what follows, we denote by $\mathbb{S}^n(r)$ the round $n$-sphere of
radius $r$, and by
$\mathbb{S}_+^n(r) := \{ x \in \mathbb{R}^{n+1} : |x| = r,\; x_{n+1} \ge 0 \}$
its upper hemisphere. Our main Riemannian results treat all three
space-form regimes $\Lambda \in \mathbb{R}$ in a unified framework and
avoid these auxiliary intrinsic assumptions on $\Sigma$: for
$\Lambda \le 0$, they are replaced by a purely extrinsic sign condition
on $\mathrm{Ric}^M(\nu,\nu) - \frac{2}{3}\Lambda$, whereas for
$\Lambda > 0$, the ambient Ricci lower bound
$\mathrm{Ric}^M \ge \frac{2}{3}\Lambda\,g$  already suffices for
rigidity, removing the symmetry and near-roundness hypotheses of
\cite{melo2024hawking}. 

This distinction reflects the different
rigidity mechanisms: in the Euclidean and
hyperbolic settings, we rely on the boundary
rigidity Theorems  \ref{rigiditybrown}, and \ref{rigidityBYhyper}   of Shi and Tam, whereas
the spherical case uses the rigidity Theorems \ref{hangwang} and \ref{hangwang3}
of Hang and Wang, which require a Ricci
curvature lower bound on the enclosed region.

Moreover, both the curvature inequalities and
rigidity hold under a new stability condition controlling only the
constant mode of the second variation, which we introduce next.
\begin{definition}[Riemannian constant-mode stability]\label{def:cm-riem}
Let $(M,g)$ be an $n$-dimensional Riemannian manifold ($n\ge 3$) and let
$\Lambda\in\mathbb{R}$. A closed, connected CMC hypersurface
$\Sigma\subset M$ with constant mean curvature $H$ with respect to a
unit normal $\nu$ is said to be \emph{$\Lambda$-constant-mode stable} if the unconstrained second
variation of area along the constant normal speed $\alpha\equiv 1$
satisfies
\begin{equation}\label{eq:cm-riem-def}
\delta_\nu^2|\Sigma|
= \int_\Sigma H^2 - |B|^2 - \mathrm{Ric}^M(\nu,\nu) \,d\mu
\ge \left(\frac{n-2}{n-1}H^2 - \frac{2}{n}\Lambda\right)|\Sigma|.
\end{equation}
\end{definition}
Note that this inequality is saturated by geodesic spheres in  space forms  of constant sectional curvature
$\frac{2\Lambda}{n(n-1)}$.

Unlike classical
variational stability \eqref{CMCstability}, which tests all mean-zero
variations $\int_\Sigma\alpha\,d\mu = 0$, constant-mode stability probes
only the orthogonal constant mode $\alpha\equiv 1$; in particular, it is
easy to check and requires no topological restriction on $\Sigma$. The two notions are independent in general. Finally, when $n=3$ and $\Lambda=0$, a
comparison with Remark~\ref{rem:CMC-reduction} shows that on a totally
geodesic slice ($K=0$) of a spacetime,
\[
\delta^2_{\vec H}|\Sigma|
= 2H^2\left(\delta_\nu^2|\Sigma| - \frac{1}{2}H^2|\Sigma|\right),
\]
so \eqref{eq:cm-riem-def} (for $H\neq 0$) coincides precisely with the
time-symmetric reduction of the Lorentzian constant-mode stability
condition $\delta^2_{\vec H}|\Sigma|\ge 0$ introduced in
Section~\ref{stabstcmc}.

\begin{theorem}\label{thm:riem-3d-unified}
Let $(M,g)$ be a $3$-dimensional Riemannian manifold with scalar curvature
$\mathrm{Sc}^M \ge 2\Lambda$ for a constant $\Lambda \in \mathbb{R}$.
Let $\Sigma \subset M$ be a closed, connected surface of constant mean
curvature $H$ with respect to a unit normal $\nu$, satisfying either
\begin{enumerate}
\item[(i)] $\Sigma$ is
$\Lambda$-constant-mode stable, or
\item[(ii)] $\Sigma$ is topologically a sphere and is variationally stable.
\end{enumerate}
Then
\begin{equation}\label{eq:riem-3d-ineq}
H^2 \le \frac{16\pi}{|\Sigma|} - \frac{4}{3}\Lambda.
\end{equation}
Suppose additionally that $\Sigma$ is the mean-convex boundary of a
relatively compact domain $\Omega \subset M$ and that equality holds in
\eqref{eq:riem-3d-ineq}:

$\bullet$ If $\Lambda \le 0$, $\Omega$ is orientable, and
$\mathrm{Ric}^M(\nu,\nu) - \frac{2}{3}\Lambda$ does not change sign
along $\Sigma$, then $\Sigma$ is isometric to a round sphere and $\Omega$
is isometric to a geodesic ball in the simply connected $3$-dimensional
space form of constant sectional curvature $\Lambda/3$ (Euclidean space
$\mathbb{R}^3$ when $\Lambda=0$, and hyperbolic space
$\mathbb{H}^3_{\Lambda/3}$ when $\Lambda<0$).

$\bullet$ If $\Lambda > 0$ and $\mathrm{Ric}^M \ge \frac{2}{3}\Lambda\,g$ on
$\Omega$, then $\Sigma$ is 
isometric to a round sphere, and $\Omega$ is isometric to a geodesic ball
in the round sphere $\mathbb{S}^3(R)$ of radius $R = \sqrt{3/\Lambda}$.
In the minimal case $H=0$ (equivalently $|\Sigma| = 12\pi/\Lambda$), this
geodesic ball is the hemisphere.
\end{theorem}
\begin{proof}
We first establish the inequality \eqref{eq:riem-3d-ineq} under either
assumption (i) or (ii) for any $\Lambda\in\mathbb{R}$, and identify the consequences of equality along $\Sigma$.

\emph{Case (i): Constant-mode stability.}
Using $|B|^2 = |\mathring B|^2 + \frac{1}{2}H^2$, the second variation
assumption reduces to
\begin{equation}\label{eq:cm-riem-unified}
\int_\Sigma |\mathring B|^2 + \mathrm{Ric}^M(\nu,\nu)
- \frac{2}{3}\Lambda\, d\mu \le 0.
\end{equation}
Recall the Gauss equation
$\mathrm{Sc}^\Sigma = \mathrm{Sc}^M - 2\mathrm{Ric}^M(\nu,\nu)
+ \frac{1}{2}H^2 - |\mathring B|^2$, where $\mathring B$ is the
trace-free second fundamental form. Integrating over $\Sigma$ and using
\eqref{eq:cm-riem-unified}, $\mathrm{Sc}^M \ge 2\Lambda$, and the
Gauss--Bonnet bound $\int_\Sigma \mathrm{Sc}^\Sigma\,d\mu \le 8\pi$, we
obtain
\begin{equation}
\begin{aligned}
\frac{1}{2}\int_\Sigma H^2\,d\mu
&= \int_\Sigma \mathrm{Sc}^\Sigma\,d\mu
   - \int_\Sigma \mathrm{Sc}^M\,d\mu
   + 2\int_\Sigma  \mathrm{Ric}^M(\nu,\nu) + |\mathring B|^2 \,d\mu
   - \int_\Sigma |\mathring B|^2\,d\mu \\
&\le 8\pi + \left(\frac{4}{3}\Lambda - 2\Lambda\right)|\Sigma|
   - \int_\Sigma |\mathring B|^2\,d\mu
 \le 8\pi - \frac{2}{3}\Lambda|\Sigma|,
\end{aligned}
\end{equation}
which yields \eqref{eq:riem-3d-ineq}. If equality holds, then
$\int_\Sigma \mathrm{Sc}^\Sigma\,d\mu = 8\pi$ (so $\Sigma$ is a
topological sphere), $|\mathring B| \equiv 0$, $\mathrm{Sc}^M = 2\Lambda$
along $\Sigma$, and
$\int_\Sigma \bigl(\mathrm{Ric}^M(\nu,\nu) - \frac{2}{3}\Lambda\bigr)\,d\mu = 0$.

\emph{Case (ii): Variationally stable sphere.}
By the uniformization theorem, $\Sigma$ is conformally equivalent to the
round sphere. Moreover, by Hersch's lemma \cite{Hersch1970} (see also
\cite{LiYau1982}), there exists a conformal map
$\varphi = (\varphi_1,\varphi_2,\varphi_3):\Sigma\to\mathbb{S}^2\subset\mathbb{R}^3$
such that $\int_\Sigma \varphi_i\,d\mu = 0$ for $i=1,2,3$. Each
coordinate function $\varphi_i$ is therefore an admissible test function
in the volume-preserving stability inequality for CMC surfaces,
\[
\int_\Sigma |\nabla^h\varphi_i|^2\,d\mu
\ge
\int_\Sigma \bigl(|B|^2 + \mathrm{Ric}^M(\nu,\nu)\bigr)\varphi_i^2\,d\mu.
\]
By conformal invariance of the Dirichlet energy,
\begin{equation}\label{confordiri}
\sum_{i=1}^{3}\int_\Sigma |\nabla^h\varphi_i|^2\,d\mu
= \sum_{i=1}^{3}\int_{\mathbb{S}^2} |\nabla x_i|^2\,d\mu_{\mathbb{S}^2}
= \sum_{i=1}^{3}-\int_{\mathbb{S}^2} x_i\Delta x_i\,d\mu_{\mathbb{S}^2}
= \sum_{i=1}^{3}2\int_{\mathbb{S}^2} x_i^2\,d\mu_{\mathbb{S}^2}
= 8\pi.
\end{equation}
Summing over $i=1,2,3$ and using $\sum_{i=1}^3 \varphi_i^2 = 1$, we obtain
\begin{equation}
8\pi
\ge \int_\Sigma |B|^2 + \mathrm{Ric}^M(\nu,\nu)\,d\mu
= \frac{1}{2}\int_\Sigma H^2\,d\mu
  + \int_\Sigma |\mathring B|^2 + \mathrm{Ric}^M(\nu,\nu)\,d\mu.
\end{equation}
Using the Gauss equation
$\mathrm{Sc}^\Sigma = \mathrm{Sc}^M - 2\mathrm{Ric}^M(\nu,\nu)
+ \frac{1}{2}H^2 - |\mathring B|^2$, together with Gauss--Bonnet
$\int_\Sigma \frac{1}{2}\mathrm{Sc}^\Sigma\,d\mu = 4\pi$ and
$\mathrm{Sc}^M \ge 2\Lambda$, this implies
\[
8\pi \ge -4\pi
+ \int_\Sigma \frac{3}{4}H^2 + \frac{1}{2}|\mathring B|^2 + \Lambda\,d\mu,
\]
and therefore
\[
\int_\Sigma H^2\,d\mu
\le 16\pi - \frac{4}{3}\Lambda|\Sigma|
  - \frac{2}{3}\int_\Sigma |\mathring B|^2\,d\mu
\le 16\pi - \frac{4}{3}\Lambda|\Sigma|,
\]
which again yields \eqref{eq:riem-3d-ineq}. If equality holds, then
$|\mathring B|\equiv 0$ and $\mathrm{Sc}^M = 2\Lambda$ along $\Sigma$.
Integrating the Gauss equation then gives
\[
\int_\Sigma 2\mathrm{Ric}^M(\nu,\nu) - \frac{4}{3}\Lambda\, d\mu
= \int_\Sigma \frac{1}{2}H^2 + \frac{2}{3}\Lambda - \mathrm{Sc}^\Sigma\,d\mu
= 0.
\]
\emph{Equality and rigidity.}
Assume now that equality holds in \eqref{eq:riem-3d-ineq}. Under either
(i) or (ii), we have established that $\Sigma$ is a topological sphere
with $|\mathring B|\equiv 0$, $\mathrm{Sc}^M = 2\Lambda$ along $\Sigma$,
and $\int_\Sigma \mathrm{Ric}^M(\nu,\nu) - \frac{2}{3}\Lambda\,d\mu = 0$.
Under either rigidity hypothesis---that
$\mathrm{Ric}^M(\nu,\nu) - \frac{2}{3}\Lambda$ does not change sign
along $\Sigma$ (when $\Lambda\le 0$), or that
$\mathrm{Ric}^M \ge \frac{2}{3}\Lambda\,g$ on $\Omega$ (when
$\Lambda>0$), it follows that $\mathrm{Ric}^M(\nu,\nu) = \frac{2}{3}\Lambda$
pointwise on $\Sigma$. Substituting $|\mathring B|=0$,
$\mathrm{Sc}^M = 2\Lambda$, and $\mathrm{Ric}^M(\nu,\nu) = \frac{2}{3}\Lambda$
into the Gauss equation yields
\[
\mathrm{Sc}^\Sigma
= \frac{2}{3}\Lambda + \frac{1}{2}H^2
= \frac{8\pi}{|\Sigma|}
= \frac{2}{r^2},
\qquad r := \sqrt{\frac{|\Sigma|}{4\pi}}.
\]
Thus $\Sigma$ is a topological sphere of constant positive Gaussian
curvature $1/r^2$, so $(\Sigma,h)$ is isometric to the round sphere of
radius $r$. We now conclude according to the sign of $\Lambda$:

\emph{Case $\Lambda = 0$.}
Here $H^2 = \frac{16\pi}{|\Sigma|} = \frac{4}{r^2} > 0$, so mean-convexity
gives $H = \frac{2}{r} > 0$. Since $(\Sigma,h)$ is isometric to the round
sphere of radius $r$, its unique convex isometric embedding into
$\mathbb{R}^3$ is the standard round sphere with mean curvature
$H_0 = \frac{2}{r} = H$. Applying Shi--Tam rigidity
(Theorem \ref{rigiditybrown}) implies that $\Omega$ is isometric to a
domain in $\mathbb{R}^3$ bounded by a round sphere, hence to a Euclidean
ball.

\emph{Case $\Lambda < 0$.}
Here $H^2 = \frac{16\pi}{|\Sigma|} - \frac{4}{3}\Lambda > 0$, so
mean-convexity again gives $H > 0$. By Theorem \ref{rigidityBYhyper},
$\Sigma$ admits a convex isometric embedding into
$\mathbb{H}^3_{\Lambda/3}$ whose image is a geodesic sphere with outward
mean curvature $H_0 > 0$ satisfying
\[
H_0^2 = -\frac{4}{3}\Lambda + \frac{16\pi}{|\Sigma|} = H^2.
\]
Hence $H_0 = H$, so $\int_\Sigma (H_0 - H)\,d\mu = 0$, and Theorem \ref{rigidityBYhyper} implies that $\Omega$ is
isometric to a geodesic ball in $\mathbb{H}^3_{\Lambda/3}$.

 \emph{Case $\Lambda > 0$.}
Let $R := \sqrt{3/\Lambda}$ be the radius of the reference space form
$\mathbb{S}^3(R)$. Substituting $\frac{2}{3}\Lambda = \frac{2}{R^2}$
into $\mathrm{Sc}^\Sigma = \frac{2}{r^2} = \frac{2}{3}\Lambda + \frac{1}{2}H^2$
and solving for $H^2$ yields
\begin{equation}
H^2 = 4\left(\frac{1}{r^2} - \frac{1}{R^2}\right),
\end{equation}
which agrees with the squared mean curvature of a geodesic sphere of
area radius $r$ in $\mathbb{S}^3(R)$. Since $H^2 \ge 0$, we have
$r \le R$. Because $\mathring B = 0$ and $\Sigma$ is mean-convex
($H \ge 0$), its second fundamental form satisfies
$B = \frac{H}{2}h \ge 0$.

Consider the rescaled manifold $(\Omega,\tilde g)$ with
$\tilde g := R^{-2}g$, which normalizes the Ricci lower bound to
$\mathrm{Ric}^{\tilde g} \ge 2\tilde g$ and maps $\Sigma$ to a round
sphere of intrinsic radius $\tilde r = r/R \le 1$. Let
$\Omega_0 \subset \mathbb{S}_+^3$ be a geodesic ball in the unit
hemisphere whose boundary $\partial\Omega_0$ is a round sphere of
intrinsic radius $\tilde r$, so there is a boundary isometry
$\iota:(\Sigma,\tilde h)\to\partial\Omega_0$. The rescaled mean
curvature $\tilde H = RH \ge 0$ satisfies
\[
\tilde H^2 = R^2 H^2 = 4\left(\frac{R^2}{r^2} - 1\right)
= 4\left(\frac{1}{\tilde r^2} - 1\right) = \tilde H_0^2,
\]
where $\tilde H_0 \ge 0$ is the outward mean curvature of
$\partial\Omega_0$ in $\mathbb{S}_+^3$; thus $\tilde H = \tilde H_0$.
Since both $\Sigma$ and $\partial\Omega_0$ are totally umbilic, their
second fundamental forms coincide under $\iota$:
$\tilde B = \frac{\tilde H}{2}\tilde h = \iota^*\tilde B_0$.

If $H > 0$, then $\tilde r < 1$, so $\partial\Omega_0$ lies strictly in
the open hemisphere and Theorem \ref{hangwang3} implies that
$(\Omega,\tilde g)$ is isometric to $\Omega_0 \subset \mathbb{S}_+^3$.
In the minimal case $H = 0$, we have $\tilde r = 1$, so
$(\Sigma,\tilde h)$ is isometric to the unit round sphere $\mathbb{S}^2$
with $\tilde B = 0$, and Theorem \ref{hangwang} implies that
$(\Omega,\tilde g)$ is isometric to the hemisphere $\mathbb{S}_+^3$.
Undoing the rescaling in either case shows that $(\Omega,g)$ is
isometric to a geodesic ball in $\mathbb{S}^3(R)$.
\end{proof}
In the hyperbolic and spherical cases, the normalization $\mathrm{Sc}^M \geq 2 \Lambda  $ is chosen to match the dominant energy condition for a spacetime with cosmological constant $\Lambda$; in the time-symmetric case, this is precisely the normalization for which the resulting inequalities and rigidity statements coincide with the positivity and rigidity of the Hawking energy with cosmological constant.

We also obtain a result for higher dimensions under constant-mode stability.
\begin{theorem}\label{thm:riem-hd-unified}
Let $(M,g)$ be a Riemannian manifold of dimension $n \ge 3$ with scalar
curvature $\mathrm{Sc}^M \ge 2\Lambda$ for a constant
$\Lambda \in \mathbb{R}$. Let $\Sigma\subset M$ be a closed,
connected, $\Lambda$-constant-mode stable
CMC hypersurface. Then
\begin{equation}\label{eq:riem-hd-ineq}
H^2 \le \frac{n-1}{n-2}\frac{\int_\Sigma \mathrm{Sc}^\Sigma\,d\mu}{|\Sigma|}
- \frac{2(n-1)}{n}\Lambda.
\end{equation}
Suppose additionally that $\Sigma$ admits an isometric embedding into
$\mathbb{R}^n$ as a convex hypersurface, that $\Sigma$ is the
mean-convex boundary of a relatively compact domain $\Omega \subset M$,
and that equality holds in \eqref{eq:riem-hd-ineq}:

 $\bullet$ If $\Lambda = 0$, $\Omega$ is orientable, and $\mathrm{Ric}^M(\nu,\nu)$
does not change sign along $\Sigma$, then $\Sigma$ is isometric to a
round sphere and $\Omega$ is isometric to a Euclidean ball in
$\mathbb{R}^n$.

 $\bullet$ If $\Lambda > 0$ and $\mathrm{Ric}^M \ge \frac{2}{n}\Lambda\,g$ on
$\Omega$, then $\Sigma$ is 
isometric to a round sphere, and $\Omega$ is isometric to a geodesic ball
in $\mathbb{S}^n(R)$, where $R = \sqrt{\frac{n(n-1)}{2\Lambda}}$. In the
minimal case $H=0$,
this geodesic ball is the hemisphere.
\end{theorem}
\begin{proof}
Using $|B|^2 = |\mathring B|^2 + \frac{1}{n-1}H^2$, the second variation
assumption reduces to
\begin{equation}\label{eq:hd-cm-cond}
\int_\Sigma |\mathring B|^2 + \mathrm{Ric}^M(\nu,\nu)
- \frac{2}{n}\Lambda \,d\mu \le 0.
\end{equation}
For $n \ge 3$, the Gauss equation is $
\mathrm{Sc}^\Sigma
= \mathrm{Sc}^M - 2\mathrm{Ric}^M(\nu,\nu)
  + \frac{n-2}{n-1}H^2 - |\mathring B|^2$. Isolating the mean curvature term, integrating over $\Sigma$, and using
$\mathrm{Sc}^M \ge 2\Lambda$ together with \eqref{eq:hd-cm-cond}, we
obtain
\begin{equation}
\begin{aligned}
\frac{n-2}{n-1}\int_\Sigma H^2\,d\mu
&= \int_\Sigma \mathrm{Sc}^\Sigma\,d\mu
   - \int_\Sigma \mathrm{Sc}^M\,d\mu
   + 2\int_\Sigma \mathrm{Ric}^M(\nu,\nu) - \frac{2}{n}\Lambda
   + |\mathring B|^2\,d\mu \\
&\quad - \int_\Sigma |\mathring B|^2\,d\mu + \frac{4}{n}\Lambda|\Sigma| \\
&\le \int_\Sigma \mathrm{Sc}^\Sigma\,d\mu
   - 2\Lambda|\Sigma| + \frac{4}{n}\Lambda|\Sigma|
 = \int_\Sigma \mathrm{Sc}^\Sigma\,d\mu
   - \frac{2(n-2)}{n}\Lambda|\Sigma|.
\end{aligned}
\end{equation}
Dividing by $\frac{n-2}{n-1}|\Sigma|$ yields \eqref{eq:riem-hd-ineq}.

Assume now that equality holds in \eqref{eq:riem-hd-ineq}. Then all
inequalities in the derivation above are saturated, which forces
$|\mathring B|^2 = 0$ and $\mathrm{Sc}^M = 2\Lambda$ along $\Sigma$, as
well as
\[
\int_\Sigma \mathrm{Ric}^M(\nu,\nu) - \frac{2}{n}\Lambda \,d\mu = 0.
\]
Under either rigidity hypothesis, that $\mathrm{Ric}^M(\nu,\nu)$ does
not change sign along $\Sigma$ when $\Lambda = 0$, or that
$\mathrm{Ric}^M \ge \frac{2}{n}\Lambda\,g$ on $\Omega$ when
$\Lambda > 0$, it follows that $\mathrm{Ric}^M(\nu,\nu) = \frac{2}{n}\Lambda$
pointwise on $\Sigma$. Substituting $|\mathring B| = 0$,
$\mathrm{Sc}^M = 2\Lambda$, and $\mathrm{Ric}^M(\nu,\nu) = \frac{2}{n}\Lambda$
into the Gauss equation shows that
\[
\mathrm{Sc}^\Sigma
= \frac{2(n-2)}{n}\Lambda + \frac{n-2}{n-1}H^2
\]
is constant on $\Sigma$. By the assumed convex isometric embedding of
$\Sigma$ into $\mathbb{R}^n$ and Ros's Constant-Scalar-Curvature Rigidity
Theorem \ref{rosrigidity}, the image of $\Sigma$ is a round sphere of area radius $r := (|\Sigma|/\omega_{n-1})^{\frac{1}{n-1}}$, where $\omega_{n-1}$ denotes the volume of the standard unit $(n-1)$-sphere. In particular, $\Sigma$ is intrinsically isometric to the round $(n-1)$-sphere of
radius $r$ and has scalar curvature
$\mathrm{Sc}^\Sigma = \frac{(n-1)(n-2)}{r^2} > 0$. We now conclude
according to the value of $\Lambda$:

\emph{Case $\Lambda = 0$.}
Here $\frac{n-2}{n-1}H^2 = \mathrm{Sc}^\Sigma = \frac{(n-1)(n-2)}{r^2} > 0$,
so mean-convexity gives $H = \frac{n-1}{r} > 0$. Since a round sphere of
radius $r$ in $\mathbb{R}^n$ is totally umbilic ($|\mathring B_0|^2 = 0$),
its outward mean curvature is $H_0 = \frac{n-1}{r} = H$. Applying
Theorem~\ref{highershitam} (together with Remark \ref{remarkhigherdim})          
implies that $\Omega$ is isometric to a domain in $\mathbb{R}^n$ bounded
by a round sphere, hence to a Euclidean ball in $\mathbb{R}^n$.

\emph{Case $\Lambda > 0$.}
Let $R := \sqrt{\frac{n(n-1)}{2\Lambda}}$ be the radius of the reference
space form $\mathbb{S}^n(R)$. Substituting
$\mathrm{Sc}^M = \frac{n(n-1)}{R^2}$,
$\mathrm{Ric}^M(\nu,\nu) = \frac{n-1}{R^2}$, and
$\mathrm{Sc}^\Sigma = \frac{(n-1)(n-2)}{r^2}$ into the Gauss equation
and solving for $H^2$ yields
\[
H^2 = (n-1)^2\left(\frac{1}{r^2} - \frac{1}{R^2}\right),
\]
which agrees with the squared mean curvature of a geodesic sphere of
area radius $r$ in $\mathbb{S}^n(R)$. Since $H^2 \ge 0$, we have
$r \le R$; moreover, $\mathring B = 0$ and mean-convexity ($H \ge 0$)
give $B = \frac{H}{n-1}h \ge 0$.

From here, the rigidity argument is identical to that of
Theorem~\ref{thm:riem-3d-unified}. Rescaling the metric by
$\tilde g := R^{-2}g$ normalizes the ambient Ricci lower bound to
$\mathrm{Ric}^{\tilde g} \ge (n-1)\tilde g$ and maps $\Sigma$ to a round
sphere of intrinsic radius $\tilde r = r/R \le 1$. Let
$\Omega_0 \subset \mathbb{S}_+^n$ be a geodesic ball in the unit
hemisphere whose boundary $\partial\Omega_0$ has intrinsic radius
$\tilde r$, with boundary isometry $\iota:(\Sigma,\tilde h)\to\partial\Omega_0$.
The rescaled mean curvature $\tilde H = RH \ge 0$ matches that of
$\partial\Omega_0$, and since both boundaries are totally umbilic,
$\tilde B = \frac{\tilde H}{n-1}\tilde h = \iota^*\tilde B_0$.
Depending on whether $H > 0$ ($\tilde r < 1$) or $H = 0$ ($\tilde r = 1$),
applying Theorem~\ref{hangwang3} or
Theorem~\ref{hangwang} respectively shows that
$(\Omega,\tilde g)$ is globally isometric to $\Omega_0$. Undoing the
scaling, $(\Omega,g)$ is globally isometric to a geodesic ball in
$\mathbb{S}^n(R)$.
\end{proof}
\begin{remark}\label{manycompo}
For $\Lambda= 0$, the rigidity conclusion
remains valid if $\Sigma$ is only one
connected component of $\partial\Omega$,
provided the remaining boundary components
have positive mean curvature and satisfy
the embeddability hypotheses of the
corresponding boundary rigidity theorem
in Appendix~\ref{appendix}. In dimension
three, the required Euclidean convex
embeddability follows from positive
Gaussian curvature.
Under these assumptions, equality in the
boundary inequality forces
$\partial\Omega$ to be connected.
\end{remark}
\begin{remark}
In related work \cite{diaz2025rigidity}, the author obtained analogous rigidity results for \emph{area-constrained Willmore surfaces}, that is, surfaces satisfying the equation
$$ 0= \lambda H + \Delta H + H|\mathring{B}|^2+ H\mathrm{Ric}^M(\nu,\nu),$$
for some constant $\lambda$. In the Euclidean reference geometry setting, this leads to the inequality
$$\int_\Sigma H^2\,d\mu \le 16\pi,$$
whose equality case gives the same rigidity conclusion as the one obtained above. Analogous inequalities and rigidity results also hold in the hyperbolic and spherical reference geometry settings. Similarly, higher-dimensional analogues are obtained.
\end{remark}

\section{Geometry of spacelike surfaces and STCMC surfaces}\label{secsetup}

In this section we introduce the geometric framework for spacelike
surfaces in Lorentzian manifolds and define the class of
\emph{spacetime constant mean curvature} (STCMC) surfaces that will be
studied throughout the paper. We first recall the null geometry of
spacelike surfaces in a four-dimensional spacetime and introduce the
associated null expansions. We then define STCMC surfaces and show
that they admit a natural variational characterization in terms of the
area functional. Finally, we derive the stability operator associated
with variations in the spacelike mean curvature direction.

Our presentation follows the standard framework used for studying surfaces in Lorentzian manifolds; see for instance~\cite{mars2012stability}.

\subsection{Spacelike surfaces and null geometry}

Let $(M,g)$ be a $4$-dimensional spacetime and let $\Sigma\subset M$
be a smooth closed spacelike surface. We denote by
\[
\Phi_\Sigma:\Sigma\hookrightarrow M
\]
the embedding of $\Sigma$ into $M$, and we will often identify
$\Sigma$ with its image.

The induced Riemannian metric on $\Sigma$ will be denoted by $h$.
At each point $p\in\Sigma$ the tangent space of the spacetime
decomposes orthogonally as $
T_pM = T_p\Sigma \oplus N_p\Sigma$. Given tangent vector fields $X,Y\in T\Sigma$, the second fundamental
form of $\Sigma$ in $(M,g)$ is defined by
\[
\chi(X,Y) := -(\nabla_X Y)^{\perp},
\]
where $\nabla$ denotes the Levi-Civita connection of $(M,g)$.
Its trace with respect to $h$ defines the mean curvature vector
\[
\vec H := \operatorname{tr}_h \chi \in N\Sigma .
\]
For a normal vector field $n\in N\Sigma$ we define the second
fundamental form in the direction $n$ by
\[
\chi_n(X,Y) = \langle \chi(X,Y),n\rangle ,
\]
and its trace
\[
\theta_n = \operatorname{tr}_h \chi_n
\]
is called the \emph{expansion along $n$}. When $n$ is null,
$\theta_n$ is referred to as the \emph{null expansion}.

The normal bundle $N\Sigma$ admits a frame consisting of two
future-directed null vector fields $\{\ell,k\}$, which we normalize by
\[
\langle \ell,k\rangle = -2 .
\]
This normalization leaves the usual boost freedom $
\ell' = f\,\ell$, 
$k' = f^{-1}k$, for any positive function $f$ on $\Sigma$. The corresponding null expansions are defined by
\[
\theta_\ell := \operatorname{tr}_h\chi_\ell,
\qquad
\theta_k := \operatorname{tr}_h\chi_k .
\]
Associated with the choice of null frame is the normal connection
one-form
\[
s_\ell(X) = -\tfrac12\langle k,\nabla_X\ell\rangle,
\qquad X\in T\Sigma .
\]
With respect to the null frame $\{\ell,k\}$ the mean curvature vector
can be written as
\[
\vec H = -\tfrac12(\theta_k\,\ell + \theta_\ell\,k),
\]
and its squared norm satisfies $
\langle \vec H,\vec H\rangle = -\theta_\ell\theta_k $.
If we restrict to the regime  $\theta_\ell\theta_k <0$, then the mean curvature vector is spacelike and we can define  the vector field
\begin{equation}\label{vecU}
    U := \frac12 \left(
\sqrt{-\frac{\theta_k}{\theta_\ell}}\,\ell
+
\sqrt{-\frac{\theta_\ell}{\theta_k}}\,k
\right).
\end{equation}
Since $\theta_\ell\theta_k<0$, the quantities under the square roots
are positive and $U$ is well defined. A direct computation shows that
\[
\langle U,U\rangle = -1,
\qquad
\langle U,\vec H\rangle = 0.
\]
Thus $U$ is a unit timelike vector field orthogonal to the spacetime
mean curvature vector.

Geometrically, the tangent space $T\Sigma$ together with the spacelike
vector $\vec H$ span a three-dimensional spacelike subspace of $TM$,
and $U$ is the unique future-directed unit timelike normal vector
orthogonal to this subspace. Moreover, $U$ is invariant under the
boost transformation, so it defines a canonical observer field associated with $\Sigma$.

\subsection{Spacetime constant mean curvature surfaces}
\begin{definition}[STCMC surfaces]
\label{def:stcmc}
A spacelike surface $\Sigma$ is called a
\emph{surface of constant spacetime mean curvature} if
\[
|\vec H|^2=-\theta_\ell \theta_k = \mathrm{const}.
\]
\end{definition}
These surfaces are referred to as \emph{STCMC surfaces}
(spacetime constant mean curvature surfaces) since they provide a
natural Lorentzian analogue of constant mean curvature (CMC)
surfaces in Riemannian geometry. In the time-symmetric case where the spacetime splits as
$M\times\mathbb R$ and the slice $M$ is totally geodesic,
one has $\theta_\ell=-\theta_k=H$, and the STCMC condition reduces to
the classical constant mean curvature condition.  These surfaces arise naturally in mathematical relativity and have been studied in several recent works. See, for example,
\cite{STCMC,kroncke2024foliations,main1local,tenan2025volume,wolff2024lellis,wolff2024effects}.

STCMC surfaces admit a simple variational interpretation.
Consider normal variations of $\Sigma$ of the form $
v = \alpha \vec H$, with
$\alpha\in C^\infty(\Sigma)$. The first variation of the area functional in this direction is 
\[
\delta_{\alpha\vec H}|\Sigma|
=
\int_\Sigma \langle \vec H,\alpha\vec H\rangle d\mu
=
\int_\Sigma \alpha\,|\vec H|^2 d\mu .
\]

Imposing the constraint
\[
\int_\Sigma \alpha\, d\mu = 0 ,
\]
we see that $\Sigma$ is a critical point of the area functional under
such variations if and only if $|\vec H|^2$ is constant.
Thus STCMC surfaces are precisely the critical points of the area
functional under infinitesimally volume-preserving variations in the
spacelike mean curvature direction, which is the Lorentzian analogue
of the classical variational characterization of CMC surfaces.

\subsection{Second variation of area}

To study the stability of STCMC surfaces we need to compute the
second variation of the area functional under such variations.
For this purpose we recall the variation formulas for the null
expansions.

Let $\xi$ be a normal variation vector field along $\Sigma$ and let $\varphi_\tau$ denote the associated flow, with deformed surfaces
$\Sigma_\tau=\varphi_\tau(\Sigma)$. Choosing a corresponding family
of null normals $\ell_\tau$ along $\Sigma_\tau$, we define the first
variation of the null expansion $\theta_\ell$ by $
\delta_\xi\theta_\ell
=
\left.\frac{d}{d\tau}\,
\varphi_\tau^{\ast}(\theta_{\ell_\tau})
\right|_{\tau=0}$. We decompose the variation field along the null frame as $
\xi|_\Sigma=\alpha\,\ell-\tfrac{\psi}{2}\,k$,
where $\alpha,\psi\in C^\infty(\Sigma)$.

The first variation of $\theta_\ell$ is given by
\begin{equation}\label{eq:variation-theta-l}
\begin{aligned}
\delta_\xi\theta_\ell
=& -\Delta_h\psi
+2\,s_\ell\cdot\nabla_h\psi
+\psi\Big(
\operatorname{div}_h s_\ell-\|s_\ell\|_h^2
+\tfrac12\theta_\ell\theta_k
+\tfrac12\mathrm{Sc}^\Sigma
-\tfrac12\mathrm{Ein}(\ell,k)
\Big)
\\
&-\alpha\Big(\mathrm{Ein}(\ell,\ell)+\|\chi_\ell\|_h^2\Big)
+\kappa_\xi\,\theta_\ell, 
\end{aligned}
\end{equation}
where  $\kappa_\xi := -\tfrac12\,\langle k,\nabla_\xi\ell_{(\xi)}\rangle
$. Interchanging the roles of $\ell$ and $k$, the variation of $\theta_k$
along $\xi=\alpha k-\tfrac{\psi}{2}\ell$ is
\begin{equation}\label{eq:variation-theta-k}
\begin{aligned}
\delta_\xi\theta_k
=& -\Delta_h\psi
-2\,s_\ell\cdot\nabla_h\psi
+\psi\Big(
-\operatorname{div}_h s_\ell-\|s_\ell\|_h^2
+\tfrac12\theta_\ell\theta_k
+\tfrac12\mathrm{Sc}^\Sigma
-\tfrac12\mathrm{Ein}(\ell,k)
\Big)
\\
&-\alpha\Big(\mathrm{Ein}(k,k)+\|\chi_k\|_h^2\Big)
-\kappa_\xi\,\theta_k .
\end{aligned}
\end{equation}
Using the variation formulas above one obtains the second variation
of the area functional in the spacelike mean curvature direction.
For variations of the form $v=\alpha\vec H$ a direct computation yields
\begin{equation}\label{secondvaria}
\begin{aligned}
\delta^2_{\alpha\vec H}|\Sigma|
=& \int_\Sigma \alpha
\Big[
\theta_k\,\Delta_h(\alpha\theta_\ell)
+\theta_\ell\,\Delta_h(\alpha\theta_k)
-2\theta_k s_\ell(\nabla^h(\alpha\theta_\ell))
+2\theta_\ell s_\ell(\nabla^h(\alpha\theta_k))
\\
&\quad
-\frac{\alpha\theta_k^2}{2}\Big(\mathrm{Ein}(\ell,\ell)+\|\chi_\ell\|_h^2\Big)
-\frac{\alpha\theta_\ell^2}{2}\Big(\mathrm{Ein}(k,k)+\|\chi_k\|_h^2\Big)
\\
&\quad
+2\alpha\theta_\ell\theta_k
\Big(\|s_\ell\|_h^2-\tfrac12\mathrm{Sc}^\Sigma
+\tfrac12\mathrm{Ein}(\ell,k)\Big)
\Big]d\mu
\\
=& \int_\Sigma \alpha\,\mathcal L\alpha\,d\mu ,
\end{aligned}
\end{equation}
By integration by parts the second variation formula can be written in the form
\begin{equation}
\delta^2_{\alpha\vec H}|\Sigma|
=
\int_\Sigma
\left(
-2\theta_\ell\theta_k\,|\nabla^h\alpha|^2
+
W\,\alpha^2
\right)d\mu ,
\end{equation}
where
\begin{equation}\label{Wterm}
    \begin{aligned}
       W =&
\theta_k\Delta_h\theta_\ell
+
\theta_\ell\Delta_h\theta_k
-
2\theta_k s_\ell(\nabla^h\theta_\ell)
+
2\theta_\ell s_\ell(\nabla^h\theta_k)
-
\frac{\theta_k^2}{2}\bigl(\mathrm{Ein}(\ell,\ell)+\|\chi_\ell\|^2\bigr)\\
&-
\frac{\theta_\ell^2}{2}\bigl(\mathrm{Ein}(k,k)+\|\chi_k\|^2\bigr)
+
2\theta_\ell\theta_k
\Big(
\|s_\ell\|^2
-\frac12\mathrm{Sc}^\Sigma
+\frac12\mathrm{Ein}(\ell,k)
\Big).
    \end{aligned}
\end{equation}
On an STCMC surface  $\Sigma$ with $\theta_\ell \theta_k\neq 0$ the potential admits a more geometric
representation. Since $\theta_\ell\theta_k$ is constant, taking the gradient of
$\log|\theta_\ell\theta_k|$ yields  
\begin{equation}\label{eq:loggrad}
\frac{\nabla^h\theta_k}{\theta_k}=\nabla^h\log|\theta_k|
=-\nabla^h\log|\theta_\ell|= -\frac{\nabla^h\theta_\ell}{\theta_\ell}.
\end{equation}
Moreover,
\begin{equation}\label{eq:loglap}
\frac{\Delta_h \theta_\ell}{\theta_\ell}
=
\Delta_h \log |\theta_\ell|
+
\|\nabla^h \log |\theta_\ell|\|_h^2,
\qquad
\frac{\Delta_h \theta_k}{\theta_k}
=
\Delta_h \log |\theta_k|
+
\|\nabla^h \log |\theta_k|\|_h^2 .
\end{equation}
Factoring out the constant product $\theta_\ell\theta_k$ and using
\eqref{eq:loggrad} and \eqref{eq:loglap}, we obtain
\begin{equation}\label{exprW}
W=\theta_\ell \theta_k
\bigg[
2\|\nabla^h \log |\theta_\ell| - s_\ell\|_h^2
-\mathrm{Sc}^\Sigma
+2\,\mathrm{Ein}(U,U)
-\frac{\theta_k}{2\theta_\ell}\|\chi_\ell\|_h^2
-\frac{\theta_\ell}{2\theta_k}\|\chi_k\|_h^2
\bigg].
\end{equation}
Here we used that the Einstein tensor terms can be expressed in terms
of the canonical timelike vector field $U$ defined in \eqref{vecU}, since
\begin{equation}\label{eq:EinUU}
2\,\mathrm{Ein}(U,U)
=
\mathrm{Ein}(\ell,k)
-\frac{\theta_k}{2\theta_\ell}\mathrm{Ein}(\ell,\ell)
-\frac{\theta_\ell}{2\theta_k}\mathrm{Ein}(k,k).
\end{equation}
Consequently, the stability operator takes the form
\begin{equation}\label{stabilityop}
\mathcal L
=
\theta_\ell\theta_k(2\Delta_h + V).
\end{equation}
where
\begin{equation}\label{ExpV}
V
 =2\|\nabla^h\log|\theta_\ell|-s_\ell\|_h^2 - \mathrm{Sc}^\Sigma + 2\,\mathrm{Ein}(U,U) - \frac{1}{2}\theta_\ell\theta_k - \frac{\theta_k}{2\theta_\ell}\|\mathring{\chi}_\ell\|_h^2 - \frac{\theta_\ell}{2\theta_k}\|\mathring{\chi}_k\|_h^2.
\end{equation}
In particular, the second variation of the area functional can be written as
\[
\delta^2_{\alpha\vec H}|\Sigma|
=
\int_\Sigma \alpha\,\mathcal L\alpha\,d\mu ,
\]
for a surface $\Sigma$ with $|\vec H|^2=-\theta_\ell \theta_k\neq 0$. Since $\mathcal L$ is a second-order elliptic operator and $\Sigma$ is
compact, its spectrum is discrete. This variational structure naturally
leads to a notion of stability for STCMC surfaces.

\subsection{Stability of STCMC surfaces}\label{stabstcmc}
Motivated by the quadratic form associated with the
operator $\mathcal L$, we introduce two complementary
notions of stability for STCMC surfaces. The first
controls the second variation of area under mean-zero
variations in the mean-curvature direction, while
the second imposes a condition on the constant
normal speed $\alpha\equiv1$.

\begin{definition}[Stability of STCMC surfaces]
\label{def:stcmc-stability}
Let $\Sigma$ be a closed spacelike STCMC surface
with $|\vec H|^2\geq0$.
\begin{enumerate}
\item[(i)]
We say that $\Sigma$ is \emph{variationally stable}
if
\[
\delta^2_{\alpha\vec H}|\Sigma|
\geq
|\vec H|^4\int_\Sigma\alpha^2\,d\mu
\]
for every $\alpha\in C^\infty(\Sigma)$ satisfying $
\int_\Sigma\alpha\,d\mu=0$.

\item[(ii)]
We say that $\Sigma$ is
\emph{constant-mode stable} if
\[
\delta^2_{\vec H}|\Sigma|\geq0.
\]
\end{enumerate}
\end{definition}
\begin{remark} We use the variation direction $\vec H$ rather than the normalized direction $\vec H/|\vec H|$. In Lorentzian geometry, $|\vec H|^2=0$ does not imply $\vec H=0$; the vector $\vec H$ may be nonzero and null. Thus $\vec H/|\vec H|$ is not defined in general, whereas $\vec H$ remains a well-defined variation direction. For STCMC surfaces, this choice preserves the standard stability notion since $|\vec H|$ is constant.  
\end{remark}
In the case  $|\vec H|\neq0$,  the  spectrum of $\mathcal L$ on
$L^2(\Sigma)$ is discrete and variational stability can  be
interpreted as a spectral bound on the quadratic form restricted to the
mean-zero subspace
\[
\mathcal H_0
=
\left\{
\alpha\in C^\infty(\Sigma):
\int_\Sigma \alpha\,d\mu=0
\right\}.
\]
More precisely,
\[
\inf_{\alpha\in\mathcal H_0\setminus\{0\}}
\frac{\int_\Sigma \alpha\,\mathcal L\alpha\,d\mu}
     {\int_\Sigma \alpha^2\,d\mu}
\ge
|\vec H|^4 .
\]
Variational stability and constant-mode stability control different parts of the spectrum of $\mathcal L$: variational stability provides a lower bound on the quadratic form on the mean-zero subspace $\mathcal H_0$, while constant-mode stability tests the constant mode $\alpha\equiv1$. In general, these two conditions are independent. 
\begin{remark}[Model operator on round spheres and relation to CMC stability]
Let $\Sigma=S_r$ be a round sphere of radius $r$
in a totally geodesic slice $\mathbb R^3$ of
Minkowski spacetime. Then
$H=\frac{2}{r}$ and $\mathrm{Sc}^\Sigma=\frac{2}{r^2}=\frac{1}{2}H^2$
are constant, so the operator $\mathcal L$ in \eqref{stabilityop}
reduces to
\[
\mathcal L
= H^2\Bigl(-2\Delta_h + \mathrm{Sc}^\Sigma - \frac{1}{2}H^2\Bigr)
= -2H^2\Delta_h.
\]
Since the first nonzero eigenvalue of $-\Delta_h$ is
$\lambda_1(-\Delta_h)=\frac{2}{r^2}$, we obtain
\[
\lambda_1(\mathcal L)
= 2H^2\lambda_1(-\Delta_h)
= \frac{16}{r^4}
= |\vec H|^4,
\]
and since $\Delta_h$ annihilates constants,
$\delta_{\vec H}^2|\Sigma|=\int_\Sigma 1\,\mathcal L(1)\,d\mu=0$.
Therefore, round spheres in $\mathbb R^3$ saturate both inequalities in
Definition~\ref{def:stcmc-stability}, showing that both stability
conditions are sharp.

For comparison, the classical Jacobi operator governing
volume-preserving stability of CMC surfaces in $\mathbb R^3$ is
$L_{\mathrm{Jac}} = -\Delta_h - \|B\|_h^2$. On a round sphere
$\|B\|_h^2=\frac{1}{2}H^2$, so
\[
-2\Delta_h = 2L_{\mathrm{Jac}} + H^2,
\qquad\text{and hence}\qquad
\mathcal L = 2H^2 L_{\mathrm{Jac}} + H^4.
\]
In particular, the threshold $|\vec H|^4=H^4$ in
Definition~\ref{def:stcmc-stability}(i) cancels the constant shift $H^4$,
so the variational stability inequality reduces to
\[
\int_\Sigma \alpha\,L_{\mathrm{Jac}}\alpha\,d\mu \ge 0
\qquad
\text{for all }\alpha\in C^\infty(\Sigma)\text{ with }\int_\Sigma \alpha\,d\mu = 0.
\]
As explained in Remark~\ref{rem:CMC-reduction}, the relation
$\mathcal L = 2H^2 L_{\mathrm{Jac}} + H^4$ holds more generally for any
CMC surface in a totally geodesic ($K=0$) spacelike hypersurface.
Consequently, in the time-symmetric setting (with $H\neq 0$), our notion
of variational stability agrees identically with classical
volume-preserving CMC stability.
\end{remark}
\begin{remark}[Directional stability]
The preceding stability notions are formulated in the mean-curvature
direction. However, when $\vec H$ is spacelike, one may more generally
consider normal directions of the form
\[
V_\beta = \vec H + \beta\,U,
\]
where $U$ is the canonical unit timelike normal defined in \eqref{vecU}
and $\beta$ is a smooth function on $\Sigma$. Since
$\langle \vec H, U\rangle = 0$, we have
$\langle \vec H, V_\beta\rangle = \langle \vec H, \vec H\rangle$, and
the first variation of area along $\alpha V_\beta$ satisfies
\[
\delta_{\alpha V_\beta}|\Sigma|
=
\int_\Sigma \alpha\langle\vec H,\vec H\rangle\,d\mu.
\]
Consequently, $\Sigma$ is STCMC if and only if it is stationary under
all such variations with $\int_\Sigma\alpha\,d\mu=0$, for any fixed
$\beta$. Although the variational characterization is independent of the choice of $\beta$, the associated second variation depends on the chosen normal direction. The mean curvature direction is geometrically distinguished: it is determined by the immersion itself, requires no auxiliary choice, and yields a stability theory that recovers classical volume-preserving CMC stability in the time-symmetric setting for nonminimal surfaces.
\end{remark}
\begin{remark}[The case $\theta_\ell\theta_k=0$]
 In this case, the second variation formula simplifies substantially: the gradient terms vanish and the second variation reduces to $
\delta^2_{\alpha\vec H}|\Sigma|
=
\int_\Sigma W\,\alpha^2\,d\mu$,
where
\[
W
=
-\frac{\theta_k^2}{2}\bigl(\mathrm{Ein}(\ell,\ell)+\|\chi_\ell\|_h^2\bigr)
-\frac{\theta_\ell^2}{2}\bigl(\mathrm{Ein}(k,k)+\|\chi_k\|_h^2\bigr).
\]
If the ambient spacetime satisfies the null energy condition, then $\mathrm{Ein}(\ell,\ell)\ge0$, $\mathrm{Ein}(k,k)\ge0$, and therefore $
W\le0$ on $\Sigma$. Consequently, if such a surface is constant-mode  or  variationally stable, then necessarily $W\equiv0$. In particular,
\begin{itemize}
\item  If $\theta_\ell=0$ ($\Sigma$ is a MOTS)  and $\theta_k \neq 0$ ,  then
$\chi_\ell = 0$ and $\mathrm{Ein}(\ell,\ell)=0$, so $\Sigma$ is an outgoing nonexpanding horizon section.

\item If $\theta_k=0$ and $\theta_\ell \neq 0$, then $\chi_k = 0$ and $\mathrm{Ein}(k,k)=0$, so  $\Sigma$ is an incoming nonexpanding horizon section.
\end{itemize}
\end{remark}
Finally, other notions of stability for spacelike surfaces have been proposed
in \cite{alaee2021stable}. Inspired by the stability theory of marginally
outer trapped surfaces (MOTS), the authors introduce a stability
condition based on the first variation of
\(
|\vec H|^2 = -\theta_\ell\theta_k
\)
along suitable directions. Under this assumption they obtain bounds
for the integral
\(
\int_\Sigma \theta_\ell\theta_k\,d\mu .
\) The stability notion introduced here is instead formulated in terms
of the second variation of the area functional in the spacetime mean
curvature direction, and can be viewed as the natural Lorentzian
analogue of the classical CMC stability condition in Riemannian
geometry.

\section{Curvature inequalities and rigidity for stable STCMC surfaces}\label{sectionSTCMC}

In this section we establish a sharp curvature inequality for stable
spacetime constant mean curvature surfaces and analyze the rigidity
case. The inequality can be viewed as a Lorentzian analogue of the
classical Christodoulou-Yau estimate for stable CMC surfaces in
Riemannian geometry.

The proof combines the stability inequality derived in the previous
section with the dominant energy condition. In the rigidity case we
exploit the positivity and rigidity properties of the
Kijowski-Liu-Yau quasi-local energy together with structural results
for maximal globally hyperbolic developments.

Throughout this section we assume that the spacetime $(M,g)$ satisfies
the \emph{dominant energy condition} (DEC), namely
\[
\mathrm{Ein}(X,Y)\ge0
\]
for all future-directed causal vectors $X,Y$.

The rigidity part of our argument also uses the 
rigidity properties of the Kijowski-Liu-Yau quasi-local energy. For a
surface $\Sigma$ with positive Gaussian curvature and spacelike mean
curvature vector, this energy is defined by
\[
\mathcal{E}_{KLY}(\Sigma)
=
\frac{1}{8\pi}
\int_\Sigma
\left(
H_0 - \sqrt{-\theta_\ell \theta_k}
\right)
d\mu,
\]
where $H_0$ is the mean curvature of the isometric embedding of
$\Sigma$ into $\mathbb{R}^3$. We will use the corresponding rigidity
theorem of Liu-Yau, Theorem~\ref{liuyaurigi}, together with the
stronger rigidity statement in Minkowski spacetime,
Theorem~\ref{minkowskirigi}; both are recalled in the appendix.

In the time-symmetric case $K=0$, the Kijowski-Liu-Yau energy reduces to the Brown-York mass (see Theorem \ref{rigiditybrown}). 

To pass from rigidity of the induced initial data to rigidity of the
ambient spacetime region, we will also use maximal globally
hyperbolic developments.

\subsection{Maximal globally hyperbolic developments}

Let $(M,g)$ be a spacetime and let $\Omega \subset M$ be a spacelike hypersurface. 
The induced Riemannian metric $g|_{\Omega}$ together with the second fundamental form $K$ of $\Omega$ in $M$ constitute an \emph{initial data set} $(\Omega,g|_{\Omega},K)$.

More generally, given an initial data set $(\Omega,g,K)$, a \emph{development} of this data consists of a spacetime $(M,\mathbf{g})$ together with an embedding
\[
\iota:\Omega \hookrightarrow M
\]
such that $\iota(\Omega)$ is a spacelike hypersurface in $(M,\mathbf{g})$ whose induced metric and second fundamental form coincide with $g$ and $K$, respectively. 
A development is called \emph{globally hyperbolic} if $\iota(\Omega)$ is a Cauchy hypersurface for $(M,\mathbf{g})$.

The Choquet-Bruhat-Geroch theorem guarantees that for every vacuum initial data set  $(\Omega,g,K)$ satisfying the Einstein constraint equations there exists a unique (up to isometry) \emph{maximal globally hyperbolic development}. 
This spacetime is maximal among all globally hyperbolic developments of the given initial data and is completely determined by the data $(g,K)$ on $\Omega$ (see, e.g., \cite{choquet2009general}).

With these preliminaries in place we can state the main curvature
inequality for stable STCMC surfaces together with the associated
rigidity statement. 
\begin{theorem}\label{thm:rigidity1}
    Let $(M,g)$ be a $4$-dimensional time-oriented Lorentzian manifold satisfying the dominant energy condition. Let $ \Sigma$ be a closed, connected spacelike STCMC surface with $|\vec{H}|^2 \geq 0$, such that either
    \begin{enumerate}
\item[(i)] $\Sigma$ is constant-mode stable ($\delta^2_{\vec H}|\Sigma|\ge 0$), or
\item[(ii)] $\Sigma$ is topologically a sphere and variationally stable
in the sense of Definition~\ref{def:stcmc-stability}.
\end{enumerate}
Then    \begin{equation}\label{inequstthe}
    |\vec{H}|^2 \leq  \frac{16\pi}{|\Sigma|}.
\end{equation} 
If equality holds in (\ref{inequstthe})
and $\mathrm{Rm}^M(k, \ell, \ell, k)$ does not change sign along $\Sigma$, then $\Sigma$ is isometric to a round sphere, and any compact mean-convex  spacelike hypersurface $\Omega \subset M$ with boundary $\partial \Omega = \Sigma$ embeds isometrically as a spacelike hypersurface in Minkowski spacetime. Furthermore, the maximal globally hyperbolic development of the induced initial data on  $\Omega$ is isometric to a standard causal diamond in Minkowski spacetime.
\end{theorem}
\begin{proof}
First note that if $|\vec H|^2=0$, then \eqref{inequstthe} is
immediate. We may therefore assume that $|\vec H|^2=-\theta_\ell\theta_k>0$.

We use the geometric formulation of the stability operator
$\mathcal L = \theta_\ell\theta_k(2\Delta_h + V) = -|\vec H|^2(2\Delta_h+V)$,
where $V$ is given by \eqref{ExpV}:
\begin{equation}\label{eq:V_tracefree}
    V = 2\|\nabla^h\log|\theta_\ell|-s_\ell\|_h^2 - \mathrm{Sc}^\Sigma + 2\,\mathrm{Ein}(U,U) - \frac{1}{2}\theta_\ell\theta_k - \frac{\theta_k}{2\theta_\ell}\|\mathring{\chi}_\ell\|_h^2 - \frac{\theta_\ell}{2\theta_k}\|\mathring{\chi}_k\|_h^2.
\end{equation}
For convenience, set
\[
I
:=
\int_\Sigma
\left(
2\|\nabla^h\log|\theta_\ell|-s_\ell\|_h^2
+2\,\mathrm{Ein}(U,U)
-\frac{\theta_k}{2\theta_\ell}\|\mathring{\chi}_\ell\|_h^2
-\frac{\theta_\ell}{2\theta_k}\|\mathring{\chi}_k\|_h^2
\right)d\mu.
\]
Because $\theta_\ell$ and $\theta_k$ have opposite signs, the ratios
$-\theta_k/\theta_\ell$ and $-\theta_\ell/\theta_k$ are strictly
positive. Furthermore, since $U$ is timelike, the dominant energy
condition ensures $\mathrm{Ein}(U,U)\ge 0$. Thus every term in the
integrand of $I$ is nonnegative, so $I\ge 0$. Integrating
\eqref{eq:V_tracefree} and using $-\frac12\theta_\ell\theta_k=\frac12|\vec H|^2$,
we obtain
\begin{equation}\label{eq:intV-decomp}
\int_\Sigma V\,d\mu
=
I - \int_\Sigma\mathrm{Sc}^\Sigma\,d\mu + \frac12|\vec H|^2|\Sigma|.
\end{equation}
\emph{Case (i): $\Sigma$ is constant-mode stable.}
By Definition~\ref{def:stcmc-stability}(ii),
\[
0 \le \delta^2_{\vec H}|\Sigma|
= \int_\Sigma 1\cdot\mathcal L(1)\,d\mu
= -|\vec H|^2\int_\Sigma V\,d\mu,
\]
and hence $\int_\Sigma V\,d\mu\le 0$. Substituting \eqref{eq:intV-decomp}
and applying the Gauss--Bonnet bound $\int_\Sigma\mathrm{Sc}^\Sigma\,d\mu\le 8\pi$
gives
\begin{equation}\label{eq:cm-defect}
|\vec H|^2|\Sigma| + 2I
\le
2\int_\Sigma\mathrm{Sc}^\Sigma\,d\mu
\le
16\pi,
\end{equation}
which yields \eqref{inequstthe}.

\emph{Case (ii): $\Sigma$ is a variationally stable topological sphere.}
Since $\Sigma$ is topologically a sphere, by the uniformization theorem
and Hersch's lemma \cite{Hersch1970} (see also \cite{LiYau1982}), there
exists a conformal map
$\varphi=(\varphi_1,\varphi_2,\varphi_3):\Sigma\to\mathbb S^2\subset\mathbb R^3$
such that $\int_\Sigma\varphi_i\,d\mu=0$ for $i=1,2,3$. Each $\varphi_i$
is therefore admissible in Definition~\ref{def:stcmc-stability}(i), so
\[
\int_\Sigma\varphi_i\mathcal L\varphi_i\,d\mu
\ge
|\vec H|^4\int_\Sigma\varphi_i^2\,d\mu,
\qquad i=1,2,3.
\]
Writing $\mathcal L=-|\vec H|^2(2\Delta_h+V)$ and dividing by
$|\vec H|^2>0$ yields
\[
\int_\Sigma\bigl(2|\nabla^h\varphi_i|^2 - V\varphi_i^2\bigr)\,d\mu
\ge
|\vec H|^2\int_\Sigma\varphi_i^2\,d\mu.
\]
Summing over $i=1,2,3$ and using $\sum_{i=1}^3\varphi_i^2=1$ together with
the conformal invariance of Dirichlet energy,
$\sum_{i=1}^3\int_\Sigma|\nabla^h\varphi_i|^2\,d\mu=8\pi$, see (\ref{confordiri}), we obtain
\[
16\pi - \int_\Sigma V\,d\mu \ge |\vec H|^2|\Sigma|.
\]
Substituting \eqref{eq:intV-decomp} and using
$\int_\Sigma\mathrm{Sc}^\Sigma\,d\mu=8\pi$ (since $\Sigma$ is a sphere)
yields
\[
I - 8\pi + \frac12|\vec H|^2|\Sigma|
\le
16\pi - |\vec H|^2|\Sigma|,
\]
or equivalently,
\begin{equation}\label{eq:var-defect}
|\vec H|^2|\Sigma| + \frac23 I \le 16\pi,
\end{equation}
which again yields \eqref{inequstthe}.

\medskip
\noindent
\emph{Equality and rigidity.}
If $|\vec H|^2=-\theta_\ell\theta_k=\frac{16\pi}{|\Sigma|}$, then
equality holds in either \eqref{eq:cm-defect} or \eqref{eq:var-defect}.
In both cases we deduce that $\int_\Sigma\mathrm{Sc}^\Sigma\,d\mu=8\pi$
(so $\Sigma$ is topologically a sphere) and $I=0$, which forces
\[
\mathrm{Ein}(U,U)=0,
\qquad
\|\mathring{\chi}_\ell\|_h^2=0,
\qquad
\|\mathring{\chi}_k\|_h^2=0,
\qquad
\|\nabla^h\log|\theta_\ell|-s_\ell\|_h^2=0
\]
along $\Sigma$. Because $U$ is a strictly timelike vector, the vanishing of $\mathrm{Ein}(U,U)$ together with the dominant energy condition implies that $\mathrm{Ein} = 0$ everywhere along $\Sigma$ (a proof of this can be found in \cite[Lemma B1]{choquet2009light}). Taking the trace of the Einstein tensor leads to $\mathrm{Sc}^M = 0$ and $\mathrm{Ric}^M = 0$ on $\Sigma$.

Since for STCMC surfaces $\nabla^h \log |\theta_\ell| = -\nabla^h\log |\theta_k|$ then we also have  $s_\ell = \nabla^h \log |\theta_\ell| = -\nabla^h\log |\theta_k|$. Note that the product $\theta_\ell \theta_k$ is boost invariant, so by choosing a boost $q := -\log |\theta_\ell|$ we obtain
$$\ell' = e^q \ell, \quad k'= e^{-q} k, \quad s_{\ell'} = s_\ell + \nabla^h q = 0, \quad \theta_{\ell'} = e^q \theta_{\ell} = \pm 1.$$
In this boosted frame, we have $s_{\ell'} = 0$ and $\theta_{k'}$ is also constant.

Recall the Gauss equation for surfaces in a $4$-dimensional Lorentzian
manifold (see, for example, \cite[equation~(9)]{andersson2010curvature}):
\begin{equation}
\mathrm{Sc}^M + 2\mathrm{Ric}^M(k,\ell) - \frac{1}{2}\mathrm{Rm}^M(k,\ell,\ell,k)
=
\mathrm{Sc}^\Sigma + \frac{1}{2}\theta_\ell\theta_k - \langle\mathring{\chi}_\ell,\mathring{\chi}_k\rangle_h.
\end{equation}
Applying our vanishing constraints, this reduces to 
\begin{equation}
-\frac{1}{2}   \mathrm{Rm}^M(k, \ell, \ell, k) =\mathrm{Sc}^\Sigma+ \frac{1}{2} \theta_\ell \theta_k. 
\end{equation}
Integrating this identity over $\Sigma$, and recalling that $\int_\Sigma \mathrm{Sc}^\Sigma \, d\mu = 8\pi$ (since $\Sigma$ is a sphere) and $-\frac{1}{2}\int_\Sigma \theta_\ell \theta_k \, d\mu = 8\pi$, we obtain
\begin{equation}
-\int_\Sigma \frac{1}{2}   \mathrm{Rm}^M(k, \ell, \ell, k) \, d \mu = \int_\Sigma \mathrm{Sc}^\Sigma \, d\mu + \frac{1}{2} \int_\Sigma \theta_\ell \theta_k \, d\mu = 8\pi - 8\pi = 0.
\end{equation}
Since $\mathrm{Rm}^M(k,\ell,\ell,k)$ does not change sign along $\Sigma$, it follows that $\mathrm{Rm}^M(k,\ell,\ell,k)=0$ pointwise on $\Sigma$.

Consequently, the Gauss equation further reduces to $\mathrm{Sc}^\Sigma = -\frac{1}{2}\theta_\ell\theta_k$, which means $\mathrm{Sc}^\Sigma$ is a positive constant. Since $\Sigma$ is a topological sphere with constant positive scalar curvature, it has constant Gaussian curvature and is therefore isometric to a round sphere of radius $r$ in Euclidean space, in particular, $
\mathrm{Sc}^\Sigma = \frac{2}{r^2}$. 
For the isometric embedding into $\mathbb{R}^3$, the mean curvature is $H_0 = \frac{2}{r}$. Moreover, since $-\theta_\ell\theta_k = |\vec H|^2 = \frac{4}{r^2}$, we obtain
\[
H_0 = \sqrt{-\theta_\ell\theta_k}.
\]
Thus, the Kijowski-Liu-Yau energy of $\Sigma$ vanishes, $\mathcal{E}_{KLY}(\Sigma)=0$.

Let $\Omega \subset M$ be any compact spacelike hypersurface region with boundary $\partial\Omega=\Sigma$, and let $(\Omega,g_\Omega,K_\Omega)$ denote the induced initial data. Since $(M,g)$ satisfies the dominant energy condition, so does $(\Omega,g_\Omega,K_\Omega)$. By Theorem~\ref{liuyaurigi}, the vanishing of $\mathcal{E}_{KLY}(\Sigma)$ implies that the spacetime is flat along the filling region $\Omega$. Furthermore, the initial data $(\Omega,g_\Omega,K_\Omega)$ embeds isometrically into Minkowski spacetime $\mathbb{R}^{3,1}$ as a compact spacelike graph $\tilde{\Omega}$ over a domain in $\mathbb{R}^3$, with boundary $\tilde{\Sigma}$. The remainder of the argument is carried out in this Minkowski ambient space. Since $\tilde{\Sigma}$ spans the compact spacelike hypersurface $\tilde{\Omega}$ in $\mathbb{R}^{3,1}$, possesses positive Gaussian curvature, and has a spacelike mean curvature vector, we can apply Theorem~\ref{minkowskirigi} \cite{miao2010geometric}. Because $\mathcal{E}_{KLY}(\tilde{\Sigma}) = 0$, this theorem dictates that $\tilde{\Sigma}$ lies entirely on a flat spatial hyperplane in $\mathbb{R}^{3,1}$. Because $\tilde{\Sigma}$ is intrinsically a round sphere lying in a hyperplane, it bounds a totally geodesic flat Euclidean ball $B=B_r$ within that same hyperplane. Crucially, since $\tilde{\Omega}$ is a spacelike graph $t=f(x)$ over the Euclidean ball $B=B_r$ with $f|_{\partial B}=0$, we have $|\nabla f|<1$, and hence $f$ is $1$-Lipschitz. Since the domain of dependence of a ball in Minkowski spacetime is given by 
\[
D(B)=\{(t,x)\in\mathbb{R}^{3,1}:\ |t|+|x|\le r\},
\]
the estimate $|f(x)|\le \operatorname{dist}(x,\partial B)=r-|x|$ guarantees that \((f(x),x)\in D(B)\) for every \(x\in B\), and therefore
$\tilde\Omega\subset D(B)$.

Moreover, we can track any future-directed causal curve $\gamma(s)=(t(s),x(s))$ in $D(B)$ by defining the function $\Phi(s)=t(s)-f(x(s))$. Because $\gamma$ is causal ($t'(s) \ge |x'(s)|$) and $\tilde{\Omega}$ is strictly spacelike ($|\nabla f| < 1$), differentiating $\Phi$ yields $\Phi'(s) = t'(s) - \nabla f(x(s)) \cdot x'(s) \ge t'(s) - |\nabla f| |x'(s)| > 0$. Thus, $\Phi(s)$ is strictly increasing. Any inextendible causal curve in $D(B)$ must originate on the past boundary (where $\Phi\le 0$) and terminate on the future boundary (where $\Phi\ge 0$). By continuity and strict monotonicity, $\Phi(s)$ must equal zero at exactly one point, meaning every such curve crosses $\tilde\Omega$ exactly once. Hence, their domains of dependence strictly coincide: $D(\tilde\Omega)=D(B)$.

Since the induced initial data on $\Omega$
coincide with those on $\tilde\Omega$,
uniqueness of the corresponding Einstein
development identifies their maximal
globally hyperbolic developments.
Consequently, the maximal development
of the induced data on $\Omega$ is
isometric to the standard Minkowski
causal diamond $D(B)$.
\end{proof}
\begin{remark}
The rigidity conclusion in Theorem~\ref{thm:rigidity1}
is formulated in terms of the maximal globally
hyperbolic development of the initial data induced
on $\Omega$, rather than directly in terms of the
ambient spacetime $M$. This distinction is necessary
because the geometry of $M$ away from $\Omega$ need
not be determined by the initial data unless the
ambient spacetime satisfies the Einstein equations.
Even in that case, the domain of
dependence $D(\Omega)\subset M$ need not coincide
with the maximal development, since the ambient
spacetime may realize only a proper portion of it.
Thus the corresponding Minkowski causal diamond
need not be realized in its entirety inside $M$.
\end{remark}
\vspace{-0.4mm}
\begin{figure}[H]
\centering
\includegraphics[width=0.58\textwidth]{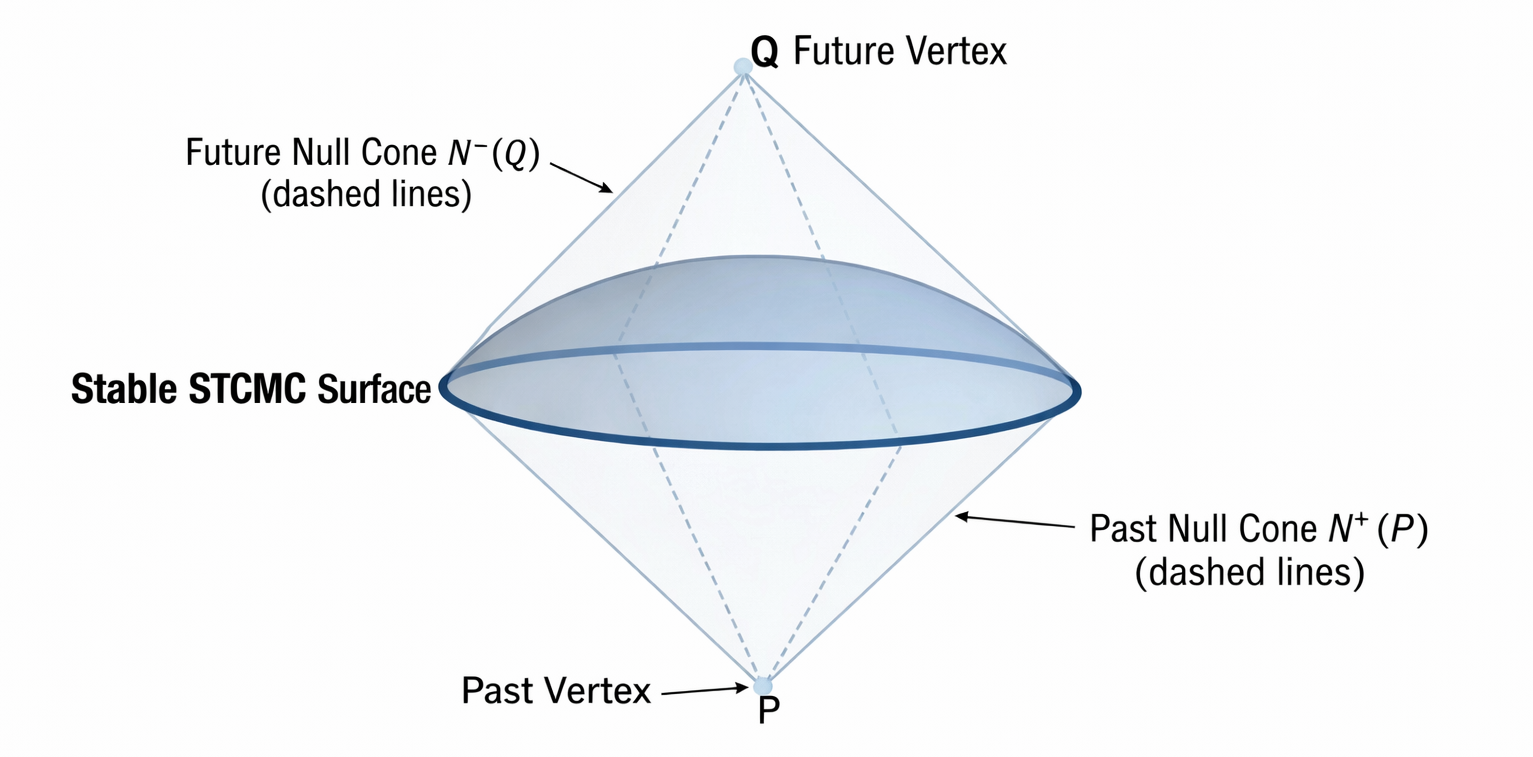}
\caption{Schematic picture of the causal diamond $D(\Omega)$ in Minkowski spacetime.
The surface $\Sigma$ is the edge of the diamond, and the shaded region is a spacelike filling.}
\end{figure}
\begin{remark}[Interpretation of the deficit]
For an STCMC surface $\Sigma$ satisfying the hypotheses of
Theorem~\ref{thm:rigidity1} with $|\vec H|^2>0$, define the deficit
\[
\mathcal D(\Sigma):=\frac{16\pi}{|\Sigma|}-|\vec H|^2 \ge 0.
\]
From \eqref{eq:cm-defect} and \eqref{eq:var-defect}, the proof of
Theorem~\ref{thm:rigidity1} yields
\[
\mathcal D(\Sigma)
\ge
\frac{c}{|\Sigma|}\int_\Sigma
2\,\mathrm{Ein}(U,U)
-\frac{\theta_k}{2\theta_\ell}\|\mathring{\chi}_\ell\|_h^2
-\frac{\theta_\ell}{2\theta_k}\|\mathring{\chi}_k\|_h^2
+2\|\nabla^h\log|\theta_\ell|-s_\ell\|_h^2
\,d\mu,
\]
where $c=2$ in case~(i) (constant-mode stability) and $c=\frac{2}{3}$
in case~(ii) (variational stability). This decomposition exhibits three distinct nonnegative contributions to
the deficit:
\begin{enumerate}
    \item \textbf{Energy density term.}
    The quantity $\mathrm{Ein}(U,U)$ is nonnegative under the dominant
    energy condition. Via the Einstein equations, it represents the
    local matter-energy density measured by the canonical timelike
    observer $U$.

    \item \textbf{Failure of umbilicity (shear).}
    The trace-free null second fundamental forms $\mathring{\chi}_\ell$
    and $\mathring{\chi}_k$ measure the failure of the surface to be
    spacetime umbilic, capturing the extrinsic deviation from a round
    sphere.

    \item \textbf{Normal bundle twist and optimal framing.}
    The one-form $s_\ell$ encodes the twist of the chosen null frame.
    The combination $\|\nabla^h\log|\theta_\ell|-s_\ell\|_h^2$ is
    boost-invariant and measures the failure of $\vec H$ to be parallel
    in the normal bundle; when it vanishes, one can boost to a frame
    with $s_{\ell'}=0$ and constant null expansions (equivalently,
    $\nabla^\perp\vec H=0$).
\end{enumerate}
This decomposition clarifies the rigidity mechanism in Theorem~\ref{thm:rigidity1}: vanishing of the deficit forces the matter
term and all geometric defect terms to vanish simultaneously, yielding
a totally umbilic sphere with parallel mean curvature vector.
\end{remark}
The rigidity statement has a useful strictness consequence: any surface
satisfying the hypotheses of Theorem~\ref{thm:rigidity1} that is not
isometric to a round sphere must satisfy \eqref{inequstthe} strictly.
Since Minkowski spacetime satisfies both the dominant energy condition
and the curvature sign assumption automatically, we obtain the following
immediate consequence.
\begin{corollary}
Let $\Sigma$ be a closed, connected spacelike STCMC surface in Minkowski
spacetime with $|\vec H|^2 \ge 0$ that is either constant-mode stable,
or topologically a sphere and variationally stable. If $\Sigma$ is not
isometric to a round sphere, then $
|\vec H|^2 < \frac{16\pi}{|\Sigma|}$.
In particular, $\Sigma$ has strictly positive Hawking energy.
\end{corollary}
\begin{remark}
Note that  the proof of the inequality $
|\vec{H}|^2   \leq  16\pi/|\Sigma|$,  under assumption $(i)$,  also holds with only  minor modification in the higher-dimensional case. Let $M$ be an $n$-dimensional Lorentzian manifold satisfying the dominant energy  condition, and $\Sigma$ a $(n-2)$-dimensional closed STCMC submanifold satisfying  $\delta^2_{\vec H}|\Sigma|\ge 0$, then
    \begin{equation}
       |\vec{H}|^2  \leq \frac{n-2}{n-3}\frac{\int_{\Sigma } \mathrm{Sc}^{\Sigma}  d\mu}{|\Sigma|} 
    \end{equation}
This, in particular, implies that the higher-dimensional Hawking energy
 \begin{equation}
 \mathcal{E}_{H,n}(\Sigma)=   \frac{1}{2(n-2)(n-3) \omega_{n-2}} 
\left( \frac{|\Sigma|}{\omega_{n-2}} \right)^{\frac{1}{n-2}}
\int_{\Sigma} 
\left( \mathrm{Sc}^{\Sigma} - \frac{n-3}{n-2} |\vec{H}|^2 \right)
d\mu
\end{equation}
is nonnegative. Here $\omega_{n-2}$ denotes the volume of the $(n-2)$-dimensional round sphere in Euclidean space.  
\end{remark}
We also have an alternative rigidity result. 
\begin{theorem}\label{thm:rigidity2}
    Let $(M,g)$ be a $4$-dimensional time-oriented Lorentzian manifold satisfying the dominant energy condition. Let $\Sigma$ be a variationally stable spacelike STCMC sphere, and assume that one of the following conditions holds:
     \begin{enumerate}
\item[(i)] $\Sigma$ has even symmetry, that is, there exists a fixed-point free isometry $\rho :\Sigma \to \Sigma$ with $\rho^2 = id$.
\item[(ii)] The Gauss curvature $K_{\Sigma}$ is sufficiently $\mathcal{C}^0$-close to $\frac{4\pi}{|\Sigma|}$, that is, there exists a constant $0<\delta_0 \ll 1$ such that $\|K_{\Sigma} -\frac{4\pi}{|\Sigma|}\|_{\mathcal{C}^0} < \delta_0$.
\end{enumerate}
    If 
    \[
     |\vec H|^2=\frac{16\pi}{|\Sigma|} ,
    \]
     then $\Sigma$ is isometric to a round sphere, and any compact mean-convex  spacelike hypersurface $\Omega \subset M$ with boundary $\partial \Omega = \Sigma$ embeds isometrically into Minkowski spacetime. Furthermore, the maximal globally hyperbolic development of the induced initial data on  $\Omega$ is isometric to a standard causal diamond in Minkowski spacetime.     
\end{theorem}
\begin{proof}
We follow the argument of Sun \cite{sun2017rigidity}, based on the El Soufi-Ilias Theorem~\ref{thm:ElSoufiIlias}.  By the equality case analysis in Theorem~\ref{thm:rigidity1}, we have 
$$\mathrm{Ein}=\mathrm{Ric}^M=0, \quad \|\mathring{\chi}_\ell\|_{h}^{2}= \|\mathring{\chi}_k\|_{h}^{2}=0, \quad \text{and}\quad s_\ell=0 $$
on $\Sigma$. In this setting, the stability operator reduces to 
\begin{equation}
  \mathcal{L}=  2\theta_\ell \theta_k \Delta_h - \theta_\ell\theta_k \mathrm{Sc}^\Sigma - \frac{1}{2} \theta_\ell^2 \theta_k^2. 
\end{equation}
Define the scaled operator
\begin{equation}
 \hat{\mathcal{L}} := -\frac{1}{2\theta_\ell \theta_k} \mathcal{L} =  - \Delta_h +\frac{\mathrm{Sc}^\Sigma}{2}+ \frac{1}{4} \theta_\ell \theta_k.   
\end{equation}
This operator has the form $-\Delta_h + q$ with $q =\frac{\mathrm{Sc}^\Sigma}{2}+ \frac{1}{4} \theta_\ell \theta_k$.   Then we can use El Soufi–Ilias Theorem~\ref{thm:ElSoufiIlias}   to estimate the second eigenvalue, obtaining that 
\begin{equation}\label{secondeigel}
    \lambda_2(\hat{\mathcal{L}})\,|\Sigma| \le 8\pi + \int_\Sigma q\, d\mu.
\end{equation}
with equality  only if $\Sigma$ admits a conformal map into the standard $\mathbb{S}^2$ whose coordinate functions are second eigenfunctions of $\hat{\mathcal{L}} $. Then applying  Gauss-Bonnet theorem, the upper bound becomes
\begin{equation}
  \lambda_2(\hat{\mathcal{L}}) |\Sigma| \leq 8 \pi + \int_\Sigma \left( \frac{\mathrm{Sc}^\Sigma}{2}+ \frac{1}{4} \theta_\ell \theta_k \right) d\mu= 8 \pi.
\end{equation}
Conversely, since $\Sigma$ is variationally stable, by the min-max principle the second eigenvalue of the original operator satisfies $\lambda_2(\mathcal{L}) \ge  |\vec{H}|^4 = \left(\frac{16\pi}{|\Sigma|} \right)^2$. Scaling this bound, we obtain
\[
\lambda_2(\hat{\mathcal{L}}) = -\frac{1}{2\theta_\ell\theta_k}\,\lambda_2(\mathcal{L}) \ge \frac{8\pi}{|\Sigma|}.
\]
Together with the upper bound above, this yields the exact equality
\[
\lambda_2(\hat{\mathcal{L}}) = \frac{8\pi}{|\Sigma|} = \frac{2}{r^2},
\]
where $r$ is the area radius of $\Sigma$.  Since we have equality in (\ref{secondeigel})   there exists a conformal map $\varphi : \Sigma \to \mathbb{S}^2 \subset \mathbb{R}^3$ such that its components $\varphi_i$ are second eigenfunctions of  $\widehat{\mathcal L}$.  

We can see that these functions have zero mean. Setting $
m_i:=\frac1{|\Sigma|}
\int_\Sigma\varphi_i\,d\mu$,
self-adjointness, $\widehat{\mathcal L}\varphi_i
=\frac2{r^2}\varphi_i$, and
$\int_\Sigma q\,d\mu=0$ give
\[
\int_\Sigma(\varphi_i-m_i)
\left(\widehat{\mathcal L}-\frac2{r^2}\right)
(\varphi_i-m_i)\,d\mu
=-8\pi m_i^2.
\]
The left-hand side is nonnegative by
variational stability, since
$\varphi_i-m_i$ has zero mean. Thus $m_i=0$. Then
$\varphi_i$ ($i=1,2,3$) satisfy $\int_\Sigma \varphi_i \, d \mu = 0$, $ \sum^3_{i=1} \varphi_i^2 =1$, and 
\begin{equation}\label{formlphi}
     -\Delta_h \varphi_i + \frac{\mathrm{Sc}^\Sigma}{2} \varphi_i + \frac{1}{4} \theta_\ell \theta_k \varphi_i -  \frac{2}{r^2}\varphi_i = 0.
\end{equation}
Since $|\varphi|^2 = \sum^3_{i=1} \varphi_i^2 = 1$, we have $0 = \Delta_h |\varphi|^2 = \sum^3_{i=1} 2 \varphi_i \Delta_h \varphi_i + 2 |\nabla^h \varphi|^2$, where $|\nabla^h \varphi|^2 = \sum^3_{i=1} |\nabla^h \varphi_i|^2$. Substituting $\Delta_h \varphi_i$ from \eqref{formlphi} yields
\begin{equation}\label{gradientvarphi}
  |\nabla^h \varphi|^2 = \frac{2}{r^2} - \frac{1}{4} \theta_\ell \theta_k - \frac{\mathrm{Sc}^\Sigma}{2} = \frac{3}{r^2} - \frac{\mathrm{Sc}^\Sigma}{2},
\end{equation}
where we used the constraint $-\theta_\ell \theta_k = \frac{16 \pi}{|\Sigma|} = \frac{4}{r^2}$. 

Now, identifying $\Sigma$ with $\mathbb{S}^2$ via the diffeomorphism $\varphi$, we can write the metric on $\Sigma$ as $h = e^u g_{\mathbb{S}^2}$, where $g_{\mathbb{S}^2}$ is the standard round metric. Because $\varphi$ is a conformal map, the trace of the pullback metric dictates that
\begin{equation}\label{conformalre}
    e^{-u} = \frac{1}{2} |\nabla^h \varphi|^2.
\end{equation}
Furthermore, by the standard conformal transformation law in two dimensions, the scalar curvature of $\Sigma$ satisfies
\begin{equation}\label{curvaturecon}
    \frac{\mathrm{Sc}^\Sigma}{2} = e^{-u} \left( 1 - \frac{1}{2} \Delta_{g_{\mathbb{S}^2}} u \right).
\end{equation}
Combining \eqref{gradientvarphi} with \eqref{conformalre} gives $\frac{\mathrm{Sc}^\Sigma}{2} = \frac{3}{r^2} - 2e^{-u}$. Inserting this into \eqref{curvaturecon} and rearranging, we find that $u$ must satisfy the partial differential equation
\begin{equation}\label{equasun}
  \Delta_{g_{\mathbb{S}^2}} u = 6 - \frac{6}{r^2} e^u.   
\end{equation}
Additionally, the coordinate center-of-mass condition implies
\begin{equation}\label{conditinsun}
    \int_{\mathbb{S}^2} x_i e^u \, d \mu_{\mathbb{S}^2} = \int_\Sigma \varphi_i e^u e^{-u} \, d \mu = \int_\Sigma \varphi_i \, d \mu = 0.
\end{equation}
Without loss of generality, we rescale the metric so that $r=1$. Equation \eqref{equasun} then becomes
\[
\Delta_{g_{\mathbb{S}^2}} u = 6 - 6 e^{u},
\]
subject to the normalization constraint \eqref{conditinsun}. This is precisely the mean-field equation considered in \cite{shi2019uniqueness, sun2017rigidity}.

Suppose first that condition $(i)$ holds. Pushing the fixed-point-free
isometric involution $\rho:(\Sigma,h)\to(\Sigma,h)$ forward via the
conformal diffeomorphism $\varphi$ yields a fixed-point-free conformal
involution $\widetilde\rho:=\varphi\circ\rho\circ\varphi^{-1}$ of
$(\mathbb{S}^2,g_{\mathbb{S}^2})$. Writing
$\widetilde\rho^*g_{\mathbb{S}^2}=e^v g_{\mathbb{S}^2}$, the isometry
condition $\widetilde\rho^*h=h$ gives
\[
u = u\circ\widetilde\rho + v.
\]
Since $\widetilde\rho^*g_{\mathbb{S}^2}=e^v g_{\mathbb{S}^2}$ has
constant Gauss curvature $1$, we have $\Delta_{g_{\mathbb{S}^2}}v=2(1-e^v)$.
On the other hand, by the conformal covariance of the two-dimensional
Laplacian and \eqref{equasun} (with $r=1$),
\[
\Delta_{g_{\mathbb{S}^2}}(u\circ\widetilde\rho)
=
e^v\bigl(\Delta_{g_{\mathbb{S}^2}}u\bigr)\circ\widetilde\rho
=
6e^v\bigl(1-e^{u\circ\widetilde\rho}\bigr)
=
6(e^v-e^u).
\]
Applying $\Delta_{g_{\mathbb{S}^2}}$ to $u=u\circ\widetilde\rho+v$
therefore yields
\[
6(1-e^u) = 6(e^v-e^u) + 2(1-e^v),
\]
which simplifies to $4(1-e^v)=0$, and hence $v\equiv 0$. Thus
$\widetilde\rho\in O(3)$ is an isometric involution of
$(\mathbb{S}^2,g_{\mathbb{S}^2})$ with no fixed points. In particular,
$+1$ cannot be an eigenvalue of $\widetilde\rho$, so all eigenvalues
must equal $-1$, which forces $\widetilde\rho(x)=-x$ to be the antipodal
map. Consequently $u(x)=u(-x)$, and
\cite[Proposition~1 and Lemma~10]{shi2019uniqueness} implies that the
only solution to \eqref{equasun} under \eqref{conditinsun} is the
trivial solution $u\equiv 0$. 

Similarly, \cite[Lemma 4]{sun2017rigidity} proves that if the Gauss curvature is sufficiently $C^0$-close to constant as required by $(ii)$, then \eqref{equasun} again admits only $u \equiv 0$.

Therefore, in either case, we conclude that $u \equiv 0$, meaning $h = g_{\mathbb{S}^2}$, and hence $\Sigma$ is isometric to the standard round sphere. In particular, $\mathrm{Sc}^\Sigma = \frac{2}{r^2} = -\frac{1}{2}\theta_\ell\theta_k$, and the remainder of the proof follows exactly as in Theorem \ref{thm:rigidity1}.
\end{proof}
\begin{remark}[Comparison with PMC rigidity results] Classical rigidity results for submanifolds in higher codimension often assume that the mean curvature vector is parallel in the normal bundle, \[ \nabla^\perp \vec H = 0, \] the so-called \emph{parallel mean curvature (PMC)} condition. Results of Chen and Yau \cite{chen1973surface,yau1974submanifolds} and subsequent extensions to Lorentzian and indefinite settings \cite{chen2013submanifolds} show that this strong pointwise assumption severely restricts the geometry of compact submanifolds. In flat ambient spaces, PMC surfaces admit a strong codimension reduction: in the non-minimal case they reduce to CMC surfaces in three-dimensional totally umbilic space forms, so that the classical spherical rigidity theory for CMC surfaces becomes available. The PMC condition implies the STCMC condition since \[ d|\vec H|^2 = 2\langle\nabla^\perp \vec H,\vec H\rangle = 0, \] but the converse is false: STCMC only fixes the length of the mean curvature vector and leaves its direction in the normal bundle unconstrained. This distinction is important in codimension two; see, for example, \cite{eschenburg1988constant}. The rigidity results of Theorems~\ref{thm:rigidity1} and \ref{thm:rigidity2} show that no PMC hypothesis is required in our setting. Instead, variational or constant-mode stability together with the dominant energy condition provides sufficient control of the normal connection, forcing the relevant square terms in the stability identity to vanish in the equality case and yielding a corresponding rigidity conclusion. 
\end{remark}
\begin{remark}[Hawking energy interpretation]
Theorems~\ref{thm:rigidity1} and \ref{thm:rigidity2} show that STCMC
surfaces are particularly well suited to the Hawking quasi-local energy (see (\ref{hawkingen})).
The Hawking energy \cite{Hawma,Living} is one of the simplest proposals
for measuring the energy contained in a bounded spacetime region and
satisfies several desirable properties, including the ADM limit,
the small-sphere limit, and monotonicity along inverse mean curvature
flow. However, it could be negative on arbitrary surfaces: even in
Euclidean space, every nonround sphere has strictly negative Hawking
energy. This highlights the need to identify geometrically distinguished
surfaces on which positivity can be recovered and the quasi-local energy
is well behaved.

The results above show that STCMC surfaces provide such a natural geometric setting. On these surfaces the Hawking energy is nonnegative under the dominant energy condition, and the rigidity statements show that vanishing energy occurs only in the flat case. Physically, this means that the Hawking energy measured on STCMC surfaces distinguishes between flat and curved spacetimes.
\end{remark}

\section{Examples and stability on STCMC foliations}\label{secfoli}

In this section, we illustrate the stability notions introduced above
through several geometric examples and applications. We first discuss
basic classes of stable STCMC surfaces and then turn to the stability
of asymptotic foliations by STCMC surfaces in both the spacelike and
null settings.

\subsection{Examples of stable STCMC surfaces}

\textbf{Round spheres in Minkowski spacetime.} 
As shown above, a round sphere lying in a totally geodesic spacelike
hyperplane of Minkowski spacetime is variationally and constant-mode
stable, with equality achieved in both conditions of
Definition~\ref{def:stcmc-stability}.

 \textbf{MOTS in stationary spacetimes.} Marginally outer trapped surfaces (MOTS) arising in stationary vacuum
spacetimes provide natural examples. A typical case is the cross-sections of the event horizon of a Kerr black hole. On such surfaces the outgoing null expansion vanishes, $
\theta_\ell \equiv 0, $
and therefore $|\vec H|^2 = -\theta_\ell\theta_k = 0.$
Stationarity implies that the horizon is shear-free in the outgoing
direction, while the vacuum Einstein equations give
$\mathrm{Ein}=0$. Substituting these relations into the stability
operator \eqref{secondvaria} shows that $\mathcal L \equiv 0$. Since the
right-hand side of Definition~\ref{def:stcmc-stability} also vanishes
when $|\vec H|^2=0$, such MOTS satisfy the STCMC stability inequalities
with equality.

\textbf{Spherically symmetric spacetimes.}
In spherically symmetric spacetimes, the round symmetry spheres $S_r$ of
area radius $r$ are STCMC surfaces. By symmetry, the null expansions
$\theta_\ell$ and $\theta_k$ are constant on each sphere, $s_\ell=0$,
and the trace-free parts $\mathring\chi_\ell$, $\mathring\chi_k$ vanish,
so the potential \eqref{ExpV} reduces to
\[
V = -\frac{2}{r^2} + 2\,\mathrm{Ein}(U,U) + \frac12|\vec H|^2.
\]
In the Schwarzschild spacetime of mass $m>0$, we have $\mathrm{Ein}=0$
and $|\vec H|^2 = \frac{4}{r^2}\bigl(1-\frac{2m}{r}\bigr)$ outside the
horizon ($r>2m$), which gives $V = -\frac{4m}{r^3}<0$. Hence
\[
\delta_{\vec H}^2|S_r|
=
\frac{64\pi m}{r^3}\left(1-\frac{2m}{r}\right) > 0,
\]
and, since the first nonzero eigenvalue of $-\Delta_{S_r}$ is
$\frac{2}{r^2}$, every nonzero mean-zero function $\alpha$ satisfies
\[
\delta^2_{\alpha\vec H}|S_r| - |\vec H|^4\int_{S_r}\alpha^2\,d\mu
\ge
|\vec H|^2\left(\frac{4}{r^2} - V - |\vec H|^2\right)\int_{S_r}\alpha^2\,d\mu
=
\frac{48m}{r^5}\left(1-\frac{2m}{r}\right)\int_{S_r}\alpha^2\,d\mu > 0.
\]
Thus all coordinate spheres outside the horizon are strictly
constant-mode and strictly variationally stable. Similarly, in the
Reissner--Nordstr\"om spacetime with $|q|\le m$, where
$|\vec H|^2=\frac{4}{r^2}\bigl(1-\frac{2m}{r}+\frac{q^2}{r^2}\bigr)$ and
$\mathrm{Ein}(U,U)=\frac{q^2}{r^4}$, we have
$V = \frac{4(q^2-mr)}{r^4}<0$ and
$\frac{4}{r^2}-V-|\vec H|^2=\frac{4(3mr-2q^2)}{r^4}>0$ for all
$r>r_+=m+\sqrt{m^2-q^2}$, so the symmetry spheres outside the outer
horizon are likewise strictly constant-mode and variationally stable.

\textbf{STCMC surfaces arising from foliations.}
The stability notions introduced here apply to asymptotic STCMC
foliations both in asymptotically Euclidean initial data sets
\cite{STCMC} and on asymptotically Schwarzschildean null hypersurfaces
\cite{kroncke2024foliations}, whose leaves are strictly constant-mode
and variationally stable when the mass is positive, see
Theorems~\ref{thm:stability-foliation} and~\ref{nullsta}. By contrast,
local STCMC foliations concentrating at a point \cite{main1local} are
generically unstable, see Theorem~\ref{localfoli}.

\subsection{STCMC surfaces on spacelike hypersurfaces}\label{STCMChyper}
$  $

Let $(\mathcal{M}^{3+1},\mathbf g)$ be a Lorentzian manifold and let
$(M,g,K)$ be a spacelike hypersurface with induced Riemannian metric $g$
and second fundamental form $K$.
Let $\Sigma \subset M$ be a smooth closed surface with induced metric $h$.
Denote by $n$ the future-pointing unit normal to $M$
and by $\nu$ the outer unit normal to $\Sigma$ in $M$.

The associated null normals are defined by
\begin{equation}
    \ell = n + \nu, \qquad
    k = n - \nu,
\end{equation}
so that $\langle \ell, k \rangle = -2$. In this setting, the null expansions are
\begin{equation}
    \theta_\ell = H + P,
    \qquad
    \theta_k = -H + P,
\end{equation}
where $P := \tr_h K = \tr K - K(\nu,\nu)$
denotes the trace of $K$ along $\Sigma$. The spacetime mean curvature satisfies
\begin{equation}
    -\theta_\ell \theta_k = H^2 - P^2.
\end{equation}
Thus STCMC surfaces generalize constant mean curvature surfaces,
reducing to the classical CMC condition in the time-symmetric case $K=0$.

We now express the stability operator (\ref{secondvaria}) in terms of initial data.
For a vector field $X$ tangent to $\Sigma$, the normal connection one-form
associated to $\ell$ is given by
\[
s_\ell(X)
=
-\frac12 \langle k, \nabla_X \ell \rangle
=
\langle \nu, \nabla_X n \rangle
=
K(X,\nu).
\]
Hence, $\|s_{\ell}\|_{h}^{2}= \|K(\cdot, \nu)\|_h^2 $. The null second fundamental forms satisfy$$\chi_\ell = K^\top + B, \qquad \chi_k = K^\top - B,$$where $B$ is the second fundamental form of $\Sigma\subset M$, and $K^\top$ is the restriction of $K$ to $T\Sigma$. 
The Einstein constraint equations on $(M,g,K)$ take the form
\begin{equation}\label{einscons}
   \mathrm{Sc}^M - |K|^2 + (\mathrm{tr}_g K)^2 = 2\,\mu, \qquad \operatorname{div}^M (K - (\mathrm{tr}_g K)g) = J, 
\end{equation}  
where the energy density $\mu$ and momentum density $J$ are defined along $M$ by$$\mu := \mathrm{Ein}(n,n), \qquad J(X) := -\mathrm{Ein}(n,X) \quad \text{for all } X \in TM.$$Here $\mathrm{Ein}$ denotes the Einstein tensor of the ambient spacetime. The dominant energy condition in this setting guarantees that $\mu \ge |J|$.  The canonical
timelike normal vector field $U$ introduced in \eqref{vecU} takes the form
\[
U=\frac{H}{\sqrt{H^2-P^2}}\,n-\frac{P}{\sqrt{H^2-P^2}}\,\nu.
\]
Consequently,
\[
2\,\mathrm{Ein}(U,U)
=
\frac{2}{H^2-P^2}
\Big(
H^2\mu+2HP\,J(\nu)+P^2\mathrm{Ein}(\nu,\nu)
\Big),
\]
substituting these identities into the geometric expression $
\mathcal L=\theta_\ell\theta_k(2\Delta_h+V)$
yields
\[
\delta^2_{\alpha\vec H}|\Sigma|
=
\int_\Sigma \alpha\,\mathcal L\alpha\,d\mu
=
 (H^2-P^2)\int_\Sigma
\left(
-2\alpha\Delta_h\alpha
-
V\,\alpha^2
\right)d\mu,
\]
where
\begin{equation}\label{Vterm}
\begin{aligned}
V=&
2\|\nabla^h\log|H+P|-K(\cdot,\nu)\|_h^2
+\frac{H-P}{2(H+P)}\|K^\top+B\|_h^2
+\frac{H+P}{2(H-P)}\|K^\top-B\|_h^2  \\
& +\frac{2}{H^2-P^2}
\Big(
H^2\mu+2HP\,J(\nu)+P^2\mathrm{Ein}(\nu,\nu)
\Big) -\mathrm{Sc}^\Sigma .
\end{aligned}
\end{equation}
\begin{remark}[Relation to classical CMC stability]
\label{rem:CMC-reduction}
Unlike the Riemannian stability condition for CMC surfaces, the stability
operator for STCMC surfaces involves the spacetime curvature component
 $\mathrm{Ein}(\nu,\nu)$, which cannot be expressed solely in
terms of the initial data $(M,g,K)$.
Consequently, STCMC stability is intrinsically a Lorentzian notion,
depending on the embedding of the initial data set into spacetime.

If we suppose that we are in a totally geodesic hypersurface, that is $K=0$, then our stability operator is 
\begin{equation}
\begin{aligned}
   \int_\Sigma \alpha  \mathcal{L} \alpha d\mu&= H^2\int_\Sigma- 2\alpha \Delta_h \alpha + \alpha^2 ( \mathrm{Sc}^\Sigma - \frac{1}{2}H^2-\|\mathring{B}\|_h^2   - \mathrm{Sc}^M ) d\mu \\
    &=2H^2\int_\Sigma- \alpha \Delta_h \alpha - \alpha^2 (  \mathrm{Ric}^M(\nu, \nu) + \|B\|_h^2 - \frac{1}{2}H^2     ) d\mu
   \end{aligned}
\end{equation}
where we used the Gauss equation $  \mathrm{Sc}^{\Sigma} = \mathrm{Sc}^{M} - 2\mathrm{Ric}^M(\nu, \nu) + \frac{1}{2}H^2 - \|\mathring{B}\|_h^2$. Recalling that the classical CMC Jacobi operator in $(M,g)$ is
$L_{\mathrm{Jac}} := -\Delta_h - \mathrm{Ric}^M(\nu,\nu) - \|B\|_h^2$,
the STCMC stability operator reduces to a scaled and shifted version of
$L_{\mathrm{Jac}}$,
\[
\mathcal{L} = 2H^2 L_{\mathrm{Jac}} + H^4.
\]
In particular, in the time-symmetric case $K=0$ with $H\neq 0$,
variational stability in Definition~\ref{def:stcmc-stability}(i) agrees
identically with classical volume-preserving CMC stability
($\int_\Sigma \alpha\,L_{\mathrm{Jac}}\alpha\,d\mu \ge 0$ for all
$\alpha\in C^\infty(\Sigma)$ with $\int_\Sigma \alpha\,d\mu = 0$).
\end{remark}

\subsection{Stability of  spacelike STCMC foliations}

Spacetime constant mean curvature surfaces play an important role in the geometry of asymptotically flat spacetimes. In particular, Cederbaum and Sakovich proved that asymptotically Euclidean initial data sets admit a canonical foliation by STCMC surfaces near infinity \cite{STCMC}. This foliation provides a geometric characterization of the center of mass of an isolated gravitational system and serves as a natural Lorentzian analogue of the classical constant mean curvature foliation constructed by Huisken and Yau in the time-symmetric case \cite{HY}. Local foliations by STCMC surfaces have also been constructed inside spacelike hypersurfaces
\cite{main1local}.

In the works cited above, a notion of stability arises from the
linearization  in normal direction of the spacetime mean curvature $\theta_\ell\theta_k$ within the initial data hypersurface.
This operator is generally not self-adjoint and is used to establish
nondegeneracy and solve the associated elliptic problem. By contrast, the stability notions considered here are those introduced in Section~\ref{stabstcmc}. As an application of our theory, we will study the stability of the leaves of these STCMC foliations.

We first recall the Big-$O$ notation. For functions $f,g$, we write
\[
f(x)=O(g(x)) \quad (x\to\infty)
\]
if there exist constants $C,\hat{\delta}>0$ such that
$|f(x)|\le C\,|g(x)|$ for all $x\ge\hat{\delta}$.
\begin{definition}[Asymptotically Euclidean initial data set]
\label{def:AE}
Let $\varepsilon\in(0,\tfrac12]$ and let $(M^3,g,K,\mu,J)$ be a smooth initial
data set for the Einstein equations. We say that $(M,g,K,\mu,J)$ is
\emph{$C^{2}_{1/2+\varepsilon}$-asymptotically Euclidean} if there exist a
compact set $B\subset M$ and a smooth coordinate chart
\[
x : M\setminus B \longrightarrow \mathbb{R}^3\setminus B_R(0)
\]
such that, in the coordinates $x=(x^1,x^2,x^3)$ with $r=|x|$, the following
estimates hold:
\begin{align}
|g_{ij}-\delta_{ij}| + r\,|\partial_k g_{ij}| + r^2|\partial_k\partial_l g_{ij}|
&= O(r^{-1/2-\varepsilon}), \\
|K_{ij}| + r\,|\partial_k K_{ij}| =O( r^{-3/2-\varepsilon}),\quad \text{and} \quad&|\mu| + |J_i| =O(r^{-3-\varepsilon}) 
\end{align}
for all $r$ sufficiently large and for all indices $i,j,k,l\in\{1,2,3\}$.
Here $\delta_{ij}$ denotes the Euclidean metric in the coordinates $x$.
\end{definition}
\begin{theorem}[Stability of the canonical STCMC foliation]
\label{thm:stability-foliation}
Let $(M^3,g,K,\mu,J)$ be a $C^2_{1/2+\varepsilon}$-asymptotically Euclidean
initial data set, for some $\varepsilon\in (0,\tfrac12]$, with non-vanishing
ADM energy $E_{ADM}\neq 0$. Let $\{\Sigma_\sigma\}_{\sigma>\sigma_0}$ be the canonical STCMC foliation constructed in \cite{STCMC}, satisfying $H^2-P^2 = 4/\sigma^2$. Assume the Einstein tensor  of the ambient spacetime into which $(M,g,K)$ embeds satisfies $|\mathrm{Ein}(\nu_\sigma,\nu_\sigma)| = O(r^{-2})$, where $r$ is the area radius and $\nu_\sigma$ is the outward normal to $\Sigma_\sigma$. 

Then, for sufficiently large $\sigma$:
\begin{itemize}
    \item If $E_{ADM} > 0$, the leaves are strictly constant-mode and variationally stable.
    \item If $E_{ADM} < 0$, the leaves are strictly constant-mode and variationally unstable.
\end{itemize}
\end{theorem}
\begin{proof}
First, we recall the decay of certain geometric quantities on each leaf $\Sigma_\sigma$ (see \cite{STCMC}): $
   \|\mathring{B} \|^2=O(r^{-3-2\varepsilon})$, $\mathrm{Sc}^M = O(r^{-3-\varepsilon})$. 
We also recall the expression of the Hawking energy in this setting
$$\mathcal{E}_H(\Sigma_\sigma)
=
\sqrt{\frac{|\Sigma_\sigma|}{16\pi}}
\left(
1 - \frac{1}{16\pi}\int_{\Sigma_\sigma} H^2-P^2d\mu 
\right) =  \frac{1}{8\pi}\sqrt{\frac{|\Sigma_\sigma|}{16\pi}}
\left(
\int_{\Sigma_\sigma}\mathrm{Sc}^{\Sigma_\sigma}- \frac{1}{2} ( H^2-P^2)d\mu 
\right) $$

By \cite[Proposition 5.6]{STCMC}, the ADM energy is related to the Hawking energy of each leaf by \begin{equation}\label{relationhaw} \mathcal E_H(\Sigma_\sigma)=E_{ADM}+O(r^{-\varepsilon}), \end{equation} and is therefore uniformly bounded. Since $H^2-P^2=4/\sigma^2$ is constant and $|\Sigma_\sigma|=4\pi r^2$, the Hawking energy of the leaf is given by 
\begin{equation}\label{hawkingenradi} 
\mathcal E_H(\Sigma_\sigma) = \frac{r}{2} \left( 1-\frac{r^2}{\sigma^2} \right). 
\end{equation} 
 Rearranging \eqref{hawkingenradi} yields $
1 - \frac{r^2}{\sigma^2} = \frac{2\mathcal E_H(\Sigma_\sigma)}{r}$. Since the leaves expand to infinity ($r \to \infty$), the right-hand side decays to zero, so $r \sim \sigma$ as $\sigma \to \infty$. Multiplying by $\sigma$ and Taylor expanding the square root yields
\begin{equation} \label{relasigr}
r=\sigma-\mathcal E_H(\Sigma_\sigma)+O(\sigma^{-1}). 
\end{equation} 
In particular, error terms of order $O(r^{-k})$ may equivalently be written as $O(\sigma^{-k})$.

Using the expression (\ref{Vterm}) for $V$, we note that the term $W$ in the stability operator given by  (\ref{exprW})  reduces under our decays to
\begin{equation}
  W=-(H^2-P^2)   V=    (H^2-P^2) \Big(   \mathrm{Sc}^{\Sigma_\sigma} -\frac{1}{2}(H^2-P^2)\Big)
   +{O}(r^{-5-\varepsilon}).
\end{equation}
We first analyze constant-mode stability.
\begin{equation}
\begin{aligned}
     \delta_{ \vec H}^2|\Sigma_\sigma|
=& \int_{\Sigma_\sigma}  W d\mu = (H^2-P^2) \int_{\Sigma_\sigma} \mathrm{Sc}^{\Sigma_\sigma} -\frac{1}{2}(H^2-P^2)d\mu
   + {O}(r^{-3-\varepsilon})\\
   =&  \frac{32\pi}{\sigma^2} \sqrt{\frac{16\pi}{|\Sigma_\sigma|}}\mathcal{E}_H(\Sigma_\sigma)+ {O}(r^{-3-\varepsilon}) = \frac{64\pi E_{ADM}}{\sigma^2 r}+ {O}(r^{-3-\varepsilon})  \\
\end{aligned}
\end{equation}
Since the leading term is of order $O(r^{-3})$ and strictly dominates
the error $O(r^{-3-\varepsilon})$, for sufficiently large $\sigma$ the
leaves are strictly constant-mode stable if $E_{ADM}>0$ and strictly
constant-mode unstable if $E_{ADM}<0$.

We now analyze variational stability (Definition~\ref{def:stcmc-stability}(i)),
whose threshold is $|\vec H|^4 = (H^2-P^2)^2 = \frac{16}{\sigma^4}$.
Consider the eigenvalues $\lambda_i$ of $-\Delta_{\Sigma_\sigma}$. By
\cite[Lemma 5.3]{STCMC}, $\lambda_0=0$,
$|\lambda_i-\frac{2}{r^2}|=O(\sigma^{-5/2-\varepsilon})$ for $i=1,2,3$,
and $\lambda_i>\frac{5}{\sigma^2}$ for $i\ge 4$. Letting $\{f_i\}_{i\ge 0}$
denote an $L^2(\Sigma_\sigma)$-orthonormal basis of corresponding
eigenfunctions, \cite[Lemma 5.4]{STCMC} provides the estimates
\begin{equation}\label{estieig}
\begin{aligned}
\left|
\lambda_i
-
\left(
\frac{2}{\sigma^2}
+
\frac{6\,\mathcal{E}_H(\Sigma_\sigma)}{\sigma^3}
+
\int_{\Sigma_\sigma} \mathrm{Ric}^M(\nu_\sigma,\nu_\sigma)\, f_i^2 \, d\mu
\right)
\right|
&= O(\sigma^{-3-\varepsilon})
,
\qquad i=1,2,3,
\\
\left|
\int_{\Sigma_\sigma} \mathrm{Ric}^M(\nu_\sigma,\nu_\sigma)\, f_i f_j \, d\mu
\right|
&=O(\sigma^{-3-\varepsilon}),
\qquad i\neq j,\quad i,j=1,2,3.
\end{aligned}
\end{equation}
Consider an arbitrary function $\alpha$ with $\int_{\Sigma_\sigma} \alpha d\mu =0$ and $\int_{\Sigma_\sigma} \alpha^2 d\mu =1$. We decompose $\alpha = \alpha^t + \alpha^d$, where $\alpha^t := \sum_{i=1}^{3} c_i f_i =\sum_{i=1}^{3} \int_{\Sigma} \alpha f_i \, d\mu  f_i$ is the translational part and $\alpha^d := \alpha-\alpha^t=\sum_{i \ge 4} c_i f_i$ is the deformational part. Note that $\sum_{i=1}^3 c_i^2= \|\alpha^t \|_{L^2(\Sigma_\sigma)}^2$ and  $\sum_{i=4} c_i^2= \|\alpha^d \|_{L^2(\Sigma_\sigma)}^2$.

 Evaluating the stability operator yields: 
\begin{equation}
    \int_{\Sigma_\sigma} \alpha \mathcal{L} \alpha \, d\mu = \int_{\Sigma_\sigma} \alpha^t \mathcal{L} \alpha^t \, d\mu + 2\int_{\Sigma_\sigma} \alpha^t \mathcal{L} \alpha^d \, d\mu + \int_{\Sigma_\sigma} \alpha^d \mathcal{L} \alpha^d \, d\mu 
\end{equation}
For the deformational part, since $(H^2-P^2) \Big(   \mathrm{Sc}^{\Sigma_\sigma} -\frac{1}{2}(H^2-P^2)\Big)=O( r^{-\frac{9}{2}-\varepsilon })  $, we have:
\begin{equation*}
\begin{aligned}
  \int_{\Sigma_\sigma} \alpha^d \mathcal{L} \alpha^d \, d\mu=&  \int_{\Sigma_\sigma} 2(H^2 -P^2) \sum_{i=4}c_i^2 \lambda_i f_i^2 \,  d\mu + \int_{\Sigma_\sigma}  (\alpha^d)^2 W \, d\mu \\
  \geq&  2(H^2 -P^2)\lambda_4 \sum_{i=4}c_i^2\int_{\Sigma_\sigma} f_i^2 \,  d\mu +\int_{\Sigma_\sigma}  (\alpha^d)^2  (H^2-P^2) \Big(   \mathrm{Sc}^{\Sigma_\sigma} -\frac{1}{2}(H^2-P^2)\Big) \, d\mu\\
  &+ O( \|\alpha^d \|_{L^2(\Sigma_\sigma)}^2r^{-5-\varepsilon })\\
  \geq& \frac{40}{\sigma^4}  \|\alpha^d \|_{L^2(\Sigma_\sigma)}^2 -  O( \|\alpha^d \|_{L^2(\Sigma_\sigma)}^2r^{-\frac{9}{2}-\varepsilon })
  \end{aligned}
\end{equation*}
Now using the estimate (\ref{estieig}) for $\lambda_i$ 
\begin{equation*}
    \begin{aligned}
        \int_{\Sigma_\sigma} \alpha^t \mathcal{L} \alpha^t \, d\mu=&   2(H^2 -P^2)  \sum_{i=1}^3 c_i^2 \lambda_i\int_{\Sigma_\sigma} f_i^2   d\mu + \int_{\Sigma_\sigma}  (\alpha^t)^2 W  d\mu \\
  =& 2(H^2 -P^2)\Big(\frac{2}{\sigma^2}
+\frac{6\,\mathcal{E}_H(\Sigma_\sigma)}{\sigma^3} \Big)  \sum_{i=1}^3 c_i^2 +2(H^2 -P^2)  \int_{\Sigma_\sigma} \mathrm{Ric}^M(\nu_\sigma,\nu_\sigma) \sum_{i=1}^3 c_i^2f_i^2  d\mu\\
&+ \int_{\Sigma_\sigma}  (\alpha^t)^2 W  d\mu  + O( \|\alpha^t \|_{L^2(\Sigma_\sigma)}^2r^{-5-\varepsilon })\\
    \end{aligned}
\end{equation*}
Using (\ref{estieig}) for the Ricci term
\begin{equation}
\begin{aligned}
      \int_{\Sigma_\sigma} \mathrm{Ric}^M(\nu_\sigma,\nu_\sigma) \sum_{i=1}^3 c_i^2f_i^2 d\mu&=    \int_{\Sigma_\sigma} \mathrm{Ric}^M(\nu_\sigma,\nu_\sigma)(\alpha^t)^2  d\mu
-  \int_{\Sigma_\sigma} \mathrm{Ric}^M(\nu_\sigma,\nu_\sigma)\sum_{i\neq j} c_if_i c_jf_j  d\mu\\
&=    \int_{\Sigma_\sigma} \mathrm{Ric}^M(\nu_\sigma,\nu_\sigma)(\alpha^t)^2  d\mu +O\bigl(\|\alpha^t\|_{L^2}^2\sigma^{-3-\varepsilon}\bigr).
\end{aligned}
\end{equation}
 Note that $ \int_{\Sigma_\sigma}  (\alpha^t)^2 W  d\mu=\int_{\Sigma_\sigma} (\alpha^t)^2  (H^2-P^2) \Big(   \mathrm{Sc}^{\Sigma_\sigma} -\frac{1}{2}H^2\Big)  d\mu  + O( \|\alpha^t \|_{L^2(\Sigma_\sigma)}^2r^{-5-\varepsilon }) $  furthermore, by the Gauss equation  $\mathrm{Sc}^{\Sigma_\sigma}+ 2\mathrm{Ric}^M(\nu_\sigma,\nu_\sigma) -\frac{1}{2}H^2=\mathrm{Sc}^{M}- \|\mathring{B} \|^2 = O(r^{-3-\varepsilon}) $ obtaining 
\begin{equation}\label{eq:trans-part}
    \begin{aligned}
        \int_{\Sigma_\sigma} \alpha^t \mathcal{L} \alpha^t \, d\mu
=& 2(H^2 -P^2)\Big(\frac{2}{\sigma^2}
+\frac{6\,\mathcal{E}_H(\Sigma_\sigma)}{\sigma^3} \Big)  \|\alpha^t \|_{L^2(\Sigma_\sigma)}^2\\& 
 +\int_{\Sigma_\sigma} (\alpha^t)^2  (H^2-P^2) \Big(   \mathrm{Sc}^{\Sigma_\sigma}+ 2\mathrm{Ric}^M(\nu_\sigma,\nu_\sigma) -\frac{1}{2}H^2\Big)  d\mu  + O( \|\alpha^t \|_{L^2(\Sigma_\sigma)}^2r^{-5-\varepsilon })\\
 =&\Big(\frac{16}{\sigma^4}
+\frac{48\,\mathcal{E}_H(\Sigma_\sigma)}{\sigma^5} \Big)  \|\alpha^t \|_{L^2(\Sigma_\sigma)}^2 + O( \|\alpha^t \|_{L^2(\Sigma_\sigma)}^2r^{-5-\varepsilon })
\end{aligned}
\end{equation}
Because $\alpha^t$ and $\alpha^d$ are orthogonal eigenfunction
expansions of $-\Delta_{\Sigma_\sigma}$, the cross term reduces to
$\int_{\Sigma_\sigma}\alpha^t\mathcal{L}\alpha^d\,d\mu=\int_{\Sigma_\sigma}\alpha^t\alpha^d W\,d\mu$.
By Cauchy--Schwarz and weighted Young's inequality,
\[
\begin{aligned}
\left|2\int_{\Sigma_\sigma}\alpha^t\mathcal{L}\alpha^d\,d\mu\right|
&\le
2\|W\|_{L^\infty}\|\alpha^t\|_{L^2}\|\alpha^d\|_{L^2}\le
\frac{12}{\sigma^4}\|\alpha^d\|_{L^2}^2
+\frac{\sigma^4}{12}\|W\|_{L^\infty}^2\|\alpha^t\|_{L^2}^2
\\
&\leq
\frac{12}{\sigma^4}\|\alpha^d\|_{L^2}^2
+O\bigl(\sigma^{-5-2\varepsilon}\bigr)\|\alpha^t\|_{L^2}^2.
\end{aligned}
\]
Combining this bound with \eqref{eq:trans-part}, the deformational
estimate, and $|\vec H|^4=\frac{16}{\sigma^4}(\|\alpha^t\|_{L^2}^2+\|\alpha^d\|_{L^2}^2)$,
we obtain
\[
\begin{aligned}
\int_{\Sigma_\sigma}\alpha\mathcal{L}\alpha\,d\mu
-|\vec H|^4\int_{\Sigma_\sigma}\alpha^2\,d\mu
&\ge
\left(\frac{48\,\mathcal{E}_H(\Sigma_\sigma)}{\sigma^5}+O(\sigma^{-5-\varepsilon})\right)\|\alpha^t\|_{L^2}^2
\\
&\quad
+\left(\frac{12}{\sigma^4}-O(\sigma^{-9/2-\varepsilon})\right)\|\alpha^d\|_{L^2}^2.
\end{aligned}
\]
If $E_{ADM}>0$, since $\mathcal{E}_H(\Sigma_\sigma)=E_{ADM}+O(\sigma^{-\varepsilon})$
 both coefficients on the right-hand side are strictly
positive for all sufficiently large $\sigma$. Therefore
$\delta_{\alpha\vec H}^2|\Sigma_\sigma|>|\vec H|^4\int_{\Sigma_\sigma}\alpha^2\,d\mu$
for every nonzero $\alpha$ with $\int_{\Sigma_\sigma}\alpha\,d\mu=0$,
which proves that the leaves are strictly variationally stable.

If $E_{ADM}<0$, choosing the admissible test
function $\alpha=f_1$ (for which $\alpha^t=f_1$, $\|\alpha^t\|_{L^2}=1$,
and $\alpha^d=0$) and using \eqref{relationhaw} in \eqref{eq:trans-part}
yields
\[
\int_{\Sigma_\sigma}f_1\mathcal{L}f_1\,d\mu
-|\vec H|^4\int_{\Sigma_\sigma}f_1^2\,d\mu
=
\frac{48 E_{ADM}}{\sigma^5}+O(\sigma^{-5-\varepsilon})
<0
\]
for all sufficiently large $\sigma$. Hence the leaves are strictly
variationally unstable when $E_{ADM}<0$.
\end{proof}
 In \cite{main1local}, the existence and uniqueness of local foliations and concentrations of STCMC surfaces were established. The main construction is based on the local CMC foliations of Ye \cite{Ye}, where the leaves are realized as perturbations of geodesic spheres.We now study the stability of these small STCMC leaves and show that, to leading order, it is governed by the local energy density and the trace-free part of the second fundamental form $K$ at the concentration point.

\begin{theorem}[Instability of small STCMC leaves]\label{localfoli}
Let $(M^3,g,K)$ be a spacelike initial data set  and let $\{\Sigma_\sigma\}_{\sigma<\sigma_0}$ be a
family of local STCMC spheres concentrating at a point $p\in M$, with
$H^2-P^2=\frac{4}{\sigma^2}$, as constructed in \cite{main1local}. Then
for all sufficiently small $\sigma$:
\begin{enumerate}
\item[(i)]   The leaves $\Sigma_\sigma$ are strictly variationally
unstable whenever $K\neq \frac{\operatorname{tr}_g K}{3}g$ at $p$. 

\item[(ii)] If $\mu \geq 0 $ at $p$, then the leaves $\Sigma_\sigma$ are strictly constant-mode
unstable, except possibly when $\mu=0$ and
$K=\frac{\operatorname{tr}_g K}{3}g$ at $p$, where $\mu$ is the energy
density of the Einstein constraint equations \eqref{einscons}.
\end{enumerate}
\end{theorem}
\begin{proof}
We first establish $(ii)$, and then deduce $(i)$ from the resulting
asymptotic expansion of the potential $V$ in \eqref{Vterm}. Since
$H^2-P^2 = 4/\sigma^2 > 0$, constant-mode instability
($\delta^2_{\vec H}|\Sigma_\sigma|<0$) is equivalent to
$\int_{\Sigma_\sigma} V\,d\mu > 0$. We begin by evaluating this integral
on the small STCMC spheres $\Sigma_\sigma$ concentrating at $p$.

By the local construction and uniqueness theorems in \cite{main1local}, the leaves can be written as normal graphs over geodesic spheres of radius $\sigma$ centered at points $c(\sigma)$ satisfying $c(\sigma)\to p$. More precisely, after identifying $T_{c(\sigma)}M$ with $\mathbb{R}^3$ via an orthonormal frame, we may write the parametrization
\begin{equation}
F_\sigma:\mathbb{S}^2\to \Sigma_{\sigma},
\qquad
F_\sigma(x)=\exp_{c(\sigma)}\bigl(\sigma(1+\psi_\sigma(x))x\bigr),
\end{equation}
where the graphical functions satisfy $\psi_\sigma=O(\sigma^2)$ in $C^3(\mathbb{S}^2)$. Under this scaling, the geometric quantities of the leaves satisfy the expansions
\[
H=\frac{2}{\sigma}+O(\sigma), \qquad \mathring{B} = O(\sigma), \qquad K=K(p)+O(\sigma).
\]
The outward unit normal is given by $\nu_\sigma(F_\sigma(x))=x+O(\sigma)$ for $x\in \mathbb{S}^2\subset T_{c(\sigma)}M$, and the area element expands as
\begin{equation}
F_\sigma^\ast d\mu_{\Sigma_\sigma}
=
\sigma^2(1+O(\sigma))\,d\omega,
\qquad
|\Sigma_\sigma|=4\pi \sigma^2(1+O(\sigma)),
\end{equation}
where $d\omega$ is the standard area element of the round sphere $\mathbb{S}^2$.

We begin by expanding the gradient term. For a small sphere, using $P=\mathrm{tr}_g K-K(\nu,\nu)$ and $\nabla_X\nu=\frac1\sigma X+O(\sigma)$ for $X\in T\Sigma_\sigma$, we obtain
$\nabla^h P = -\frac{2}{\sigma}K(\nu,\cdot)^\top+O(1)$.
Since $\Sigma_\sigma$ is STCMC, $H^2-P^2=\mathrm{const}$, taking a tangential derivative gives $H\nabla^h H=P\nabla^h P$, and hence
\[
\nabla^h\log|H+P| = \frac{\nabla^h H+\nabla^h P}{H+P} = \frac{1}{H}\nabla^h P = -K(\nu,\cdot)^\top+O(\sigma).
\]
Consequently,
\[
2\left\| \nabla^h\log|H+P|-K(\nu,\cdot)^\top \right\|_h^2
= 8\|K(\nu,\cdot)^\top\|_h^2+O(\sigma).
\]
Next, since $B=\frac{H}{2}h+O(\sigma)$, we have
$\langle B,K^\top\rangle_h=\frac{H}{2}P+O(\sigma)$,
$\|B\|_h^2=\frac{1}{2}H^2+O(\sigma^2)$. Also, since $P/H=O(\sigma)$
\[
\frac{H\pm P}{2(H\mp P)}= \frac{1 \pm \frac{P}{H}}{2\bigl(1 \mp \frac{P}{H}\bigr)} = \frac12\pm \frac{P}{H}+\frac{P^2}{H^2}+O(\sigma^3).
\]
Therefore,
\[
\frac{H-P}{2(H+P)}\|K^\top+B\|_h^2 + \frac{H+P}{2(H-P)}\|K^\top-B\|_h^2 = \|B\|_h^2+\|K^\top\|_h^2-P^2+O(\sigma).
\]
The matter term satisfies
\[
\frac{2}{H^2-P^2} \left( H^2\mu+2HPJ(\nu)+P^2\mathrm{Ein}(\nu,\nu) \right) = 2\mu+O(\sigma).
\]
Combining the previous expansions gives
\[
V = 8\|K(\nu,\cdot)^\top\|_h^2 -\mathrm{Sc}^\Sigma +\|B\|_h^2 +\|K^\top\|_h^2 -P^2 +2\mu +O(\sigma).
\]
Since $\|B\|_h^2=\frac{1}{2}H^2+O(\sigma^2)$ and $\|K^\top\|_h^2-\frac{1}{2}P^2=\left\|K^\top-\frac{P}{2}h\right\|_h^2$, this can be written equivalently as
\begin{equation}\label{eq:V-tracefree-local}
V = -\mathrm{Sc}^\Sigma + \frac{1}{2}(H^2-P^2) + 8\|K(\nu,\cdot)^\top\|_h^2 + \left\|K^\top-\frac{P}{2}h\right\|_h^2 + 2\mu + O(\sigma).
\end{equation}
On the other hand, applying the Gauss equation to the expression for $V$ above yields
\[
V = -\mathrm{Sc}^M +2\mathrm{Ric}^M(\nu,\nu) +8\|K(\nu,\cdot)^\top\|_h^2 +\|K^\top\|_h^2 -P^2 +2\mu +O(\sigma).
\]
Using $|K|^2=\|K^\top\|_h^2+2\|K(\nu,\cdot)^\top\|_h^2+K(\nu,\nu)^2$ in the expression for $\mu$ \eqref{einscons}, we have
\begin{equation}
V = 2\mathrm{Ric}^M(\nu,\nu) +2P\,K(\nu,\nu) +6\|K(\nu,\cdot)^\top\|_h^2 +O(\sigma).
\end{equation}
Note that in this graphical setting, if a function $f_\sigma$ on
$\Sigma_\sigma$ has leading term $f_0(x)$, then
\begin{equation}
\frac{1}{|\Sigma_\sigma|}\int_{\Sigma_\sigma}f_\sigma\,d\mu_{\Sigma_\sigma}
=
\frac{1}{4\pi}\int_{\mathbb{S}^2}f_0(x)\,d\omega
+
O(\sigma),
\end{equation}
where in the last integral $x\in \mathbb{S}^2\subset T_pM$ after parallel
transporting the frame from $c(\sigma)$ to $p$.

To leading order, the average over $\Sigma_\sigma$ corresponds to the normalized integral over the unit sphere $\mathbb{S}^2\subset T_pM$. We use
\[
\frac{1}{4\pi}\int_{\mathbb{S}^2}\nu^i\nu^j\,d\omega = \frac13\delta^{ij},
\qquad
\frac{1}{4\pi}\int_{\mathbb{S}^2}\nu^i\nu^j\nu^k\nu^\ell\,d\omega = \frac1{15} \left( \delta^{ij}\delta^{k\ell} + \delta^{ik}\delta^{j\ell} + \delta^{i\ell}\delta^{jk} \right).
\]
Hence
\[
\frac{1}{4\pi}\int_{\mathbb{S}^2}2\mathrm{Ric}^M(\nu,\nu)\,d\omega = \frac{2}{3}\mathrm{Sc}^M,
\qquad
\frac{1}{4\pi}\int_{\mathbb{S}^2}K(\nu,\nu)\,d\omega=\frac13\mathrm{tr}_g K,
\]
and
\[
\frac{1}{4\pi}\int_{\mathbb{S}^2}K(\nu,\nu)^2\,d\omega = \frac{1}{15}\left((\mathrm{tr}_g K)^2+2|K|^2\right),
\qquad
\frac{1}{4\pi}\int_{\mathbb{S}^2}2P\,K(\nu,\nu)\,d\omega= \frac{8}{15}(\mathrm{tr}_g K)^2-\frac{4}{15}|K|^2.
\]
Finally, since $\|K(\nu,\cdot)^\top\|_h^2=|K(\nu,\cdot)|^2-K(\nu,\nu)^2$ and $\frac{1}{4\pi}\int_{\mathbb{S}^2}|K(\nu,\cdot)|^2\,d\omega = \frac13|K|^2$, we obtain
\[
\frac{1}{4\pi}\int_{\mathbb{S}^2}\|K(\nu,\cdot)^\top\|_h^2\,d\omega = \frac{1}{15} \left( 3|K|^2-(\mathrm{tr}_g K)^2 \right) = \frac{1}{5}|\mathring{K}|^2,
\]
where $|\mathring K|^2 = |K|^2-\frac{1}{3}(\mathrm{tr}_g K)^2$. Thus, the average of $V$ over $\Sigma_\sigma$ evaluates to
\begin{equation}
\frac{1}{|\Sigma_\sigma|}\int_{\Sigma_\sigma}V\,d\mu = \frac{2}{3}\mathrm{Sc}^M + \frac{2}{15}(\mathrm{tr}_g K)^2 + \frac{14}{15}|K|^2 + O(\sigma).
\end{equation}
Using once more the constraint equation $\mathrm{Sc}^M-|K|^2+(\mathrm{tr}_g K)^2=2\mu$, we arrive at
\begin{equation}\label{eq:V-avg-local}
\frac{1}{|\Sigma_\sigma|}\int_{\Sigma_\sigma}V\,d\mu = \frac{4}{3}\mu(p) + \frac{8}{5} |\mathring{K}|^2(p) + O(\sigma).
\end{equation}
In particular, $|\mathring K|^2\geq 0$, with equality if and only if $K=\frac{\mathrm{tr}_g K}{3}g$. Therefore, assuming $\mu\geq 0$,  the leading averaged potential is nonnegative and strictly positive if either $\mu(p)>0$ or $K$ is not pure trace at $p$. Since $H^2-P^2 > 0$, this strict positivity implies $\delta^2_{\vec H}|\Sigma_\sigma|<0$, proving $(ii)$.

To prove $(i)$, we apply Hersch's lemma \cite{Hersch1970} to the topological sphere $\Sigma_\sigma$, which yields a conformal map $\varphi=(\varphi_1,\varphi_2,\varphi_3):\Sigma_\sigma\to\mathbb{S}^2\subset\mathbb{R}^3$ satisfying $\int_{\Sigma_\sigma}\varphi_i\,d\mu=0$ for $i=1,2,3$, $\sum_{i=1}^3\varphi_i^2=1$, and $\sum_{i=1}^3\int_{\Sigma_\sigma}|\nabla^h\varphi_i|^2\,d\mu=8\pi$. Summing the variational stability excess over $i=1,2,3$ gives
\[
\frac{1}{H^2-P^2}\sum_{i=1}^3
\left(
\delta^2_{\varphi_i\vec H}|\Sigma_\sigma|
-(H^2-P^2)^2\int_{\Sigma_\sigma}\varphi_i^2\,d\mu
\right)
=
16\pi - \int_{\Sigma_\sigma}V\,d\mu - (H^2-P^2)|\Sigma_\sigma|.
\]
Integrating \eqref{eq:V-tracefree-local} over $\Sigma_\sigma$ and applying the Gauss--Bonnet theorem ($\int_{\Sigma_\sigma}\mathrm{Sc}^\Sigma\,d\mu=8\pi$) yields
\[
(H^2-P^2)|\Sigma_\sigma| - 16\pi
=
2\int_{\Sigma_\sigma}\left(V - 8\|K(\nu,\cdot)^\top\|_h^2 - \left\|K^\top-\frac{P}{2}h\right\|_h^2 - 2\mu\right)d\mu
+ O(\sigma)|\Sigma_\sigma|.
\]
Substituting this identity into the Hersch sum and dividing by $|\Sigma_\sigma|$, we obtain
\begin{equation}\label{eq:local-hersch-excess}
\begin{aligned}
&\frac{1}{(H^2-P^2)|\Sigma_\sigma|}\sum_{i=1}^3
\left(
\delta^2_{\varphi_i\vec H}|\Sigma_\sigma|
-(H^2-P^2)^2\int_{\Sigma_\sigma}\varphi_i^2\,d\mu
\right)
\\
&\qquad =
\frac{1}{|\Sigma_\sigma|}\int_{\Sigma_\sigma}
\left(16\|K(\nu,\cdot)^\top\|_h^2 + 2\left\|K^\top-\frac{P}{2}h\right\|_h^2 + 4\mu - 3V\right)d\mu
+ O(\sigma).
\end{aligned}
\end{equation}
Using $|K|^2=\|K^\top\|_h^2+2\|K(\nu,\cdot)^\top\|_h^2+K(\nu,\nu)^2$ and $P=\mathrm{tr}_g K-K(\nu,\nu)$, we rewrite
\[
2\left\|K^\top-\frac{P}{2}h\right\|_h^2
=
2\|K^\top\|_h^2 - P^2
=
2|K|^2 - 4\|K(\nu,\cdot)^\top\|_h^2 - 3K(\nu,\nu)^2 - (\mathrm{tr}_g K)^2 + 2(\mathrm{tr}_g K)K(\nu,\nu).
\]
Substituting the spherical averages already computed above yields immediately
\[
\frac{1}{4\pi}\int_{\mathbb{S}^2}2\left\|K^\top-\frac{P}{2}h\right\|_h^2\,d\omega
=
\frac{4}{5}|K|^2 - \frac{4}{15}(\mathrm{tr}_g K)^2
=
\frac{4}{5}|\mathring{K}|^2.
\]
Combining this with $\frac{1}{4\pi}\int_{\mathbb{S}^2}16\|K(\nu,\cdot)^\top\|_h^2\,d\omega=\frac{16}{5}|\mathring{K}|^2$ and \eqref{eq:V-avg-local} in \eqref{eq:local-hersch-excess}, the $\mu(p)$ terms cancel and we find
\[
\frac{1}{(H^2-P^2)|\Sigma_\sigma|}\sum_{i=1}^3
\left(
\delta^2_{\varphi_i\vec H}|\Sigma_\sigma|
-(H^2-P^2)^2\int_{\Sigma_\sigma}\varphi_i^2\,d\mu
\right)
=
-\frac{4}{5}|\mathring{K}|^2(p) + O(\sigma).
\]
Whenever $K\neq \frac{\mathrm{tr}_g K}{3}g$ at $p$, we have $|\mathring{K}|^2(p)>0$, so the sum is strictly negative for all sufficiently small $\sigma$. At least one mean-zero test function $\varphi_i$ therefore violates Definition~\ref{def:stcmc-stability}(i), proving that $\Sigma_\sigma$ is strictly variationally unstable.
\end{proof}
The assumption $\mu(p)\geq0$ in $(ii)$ can
be weakened. Indeed, the proof shows that $
\mu(p)\geq-\frac65|\mathring K|^2(p)$
is necessary for constant-mode instability,
while strict inequality is sufficient.
We retain $\mu(p)\geq0$, as it is consistent
with the dominant energy condition and yields
the geometrically distinguished borderline
case $\mu(p)=0$ and $\mathring K(p)=0$.

This result further illustrates that, in general,
variational stability and constant-mode stability
are independent notions.

\subsection{STCMC surfaces on null hypersurfaces}

In \cite{kroncke2024foliations}, Kröncke and Wolff construct an
asymptotic foliation by STCMC surfaces on asymptotically
Schwarzschildean lightcones. This provides a genuinely Lorentzian
family of spacelike STCMC surfaces, since the leaves foliate a null
hypersurface rather than a spacelike initial data set. Moreover, the foliation is unique within a suitable
a-priori class. We now show that the leaves of this foliation also
satisfy our notions of constant-mode and variational stability.

\begin{theorem}[Stability in the null STCMC foliation]\label{nullsta}
Let $\mathcal N$ be an asymptotically Schwarzschildean lightcone of
positive mass $m>0$, and let $\{\Sigma_\sigma\}_{\sigma\ge\sigma_0}$
be the asymptotic foliation by STCMC surfaces constructed in
\cite{kroncke2024foliations}. Then, for $\sigma$ sufficiently large,
the leaves $\Sigma_\sigma$ are strictly constant-mode stable and
strictly variationally stable.   
\end{theorem}
\begin{proof}
We combine the geometric stability operator \eqref{stabilityop} with the
asymptotic estimates for the null STCMC foliation obtained in
\cite{kroncke2024foliations}.

We first prove constant-mode stability. Since
$\mathcal L=\theta_\ell\theta_k(2\Delta_h+V)$ and
$-\theta_\ell\theta_k=|\vec H|^2>0$, strict constant-mode stability is
equivalent to $\int_{\Sigma_\sigma}V\,d\mu<0$. Using
$-\frac12\theta_\ell\theta_k=\frac12|\vec H|^2$ together with the
asymptotics of the foliation from \cite[Lemma~2.14]{kroncke2024foliations},
the potential \eqref{ExpV} satisfies
\begin{equation}\label{eq:null-potential-asymptotic}
\begin{aligned}
V
&=
2\|\nabla^h\log|\theta_\ell|-s_\ell\|_h^2
-\mathrm{Sc}^{\Sigma_\sigma}
+2\,\mathrm{Ein}(U,U)
+\frac12|\vec H|^2
-\frac{\theta_k}{2\theta_\ell}\|\mathring{\chi}_\ell\|_h^2
-\frac{\theta_\ell}{2\theta_k}\|\mathring{\chi}_k\|_h^2
\\
&=
\frac12|\vec H|^2
-\mathrm{Sc}^{\Sigma_\sigma}
+O(\sigma^{-4}).
\end{aligned}
\end{equation}
Moreover, the leaves satisfy
\begin{equation}\label{eq:null-leaf-asymptotics}
|\vec H|^2
=
\frac{4}{\sigma^2}
-\frac{8m}{\sigma^3}
+O(\sigma^{-4}),
\qquad
\mathrm{Sc}^{\Sigma_\sigma}
=
\frac{2}{\sigma^2}
+O(\sigma^{-4}),
\qquad
|\Sigma_\sigma|=4\pi\sigma^2.
\end{equation}
Therefore, since $\Sigma_\sigma$ is a topological sphere,
Gauss--Bonnet together with \eqref{eq:null-potential-asymptotic} and
\eqref{eq:null-leaf-asymptotics} yields
\[
-\int_{\Sigma_\sigma}V\,d\mu
=
8\pi
-\frac12
\left(
\frac{4}{\sigma^2}
-\frac{8m}{\sigma^3}
+O(\sigma^{-4})
\right)(4\pi\sigma^2)
+O(\sigma^{-2})
=
\frac{16\pi m}{\sigma}+O(\sigma^{-2}).
\]
Since $m>0$, the right-hand side is strictly positive for $\sigma$
sufficiently large. Hence $\int_{\Sigma_\sigma}V\,d\mu<0$, proving
strict constant-mode stability.

We now turn to variational stability. Let
$\alpha\in C^\infty(\Sigma_\sigma)$ be any nonzero function satisfying
$\int_{\Sigma_\sigma}\alpha\,d\mu=0$. Dividing
$\delta^2_{\alpha\vec H}|\Sigma_\sigma|>|\vec H|^4\int_{\Sigma_\sigma}\alpha^2\,d\mu$
by $|\vec H|^2>0$, strict variational stability is equivalent to
\begin{equation}\label{eq:target-variational}
\int_{\Sigma_\sigma}
\left(
2|\nabla^h\alpha|^2-V\alpha^2
\right)d\mu
>
|\vec H|^2
\int_{\Sigma_\sigma}\alpha^2\,d\mu.
\end{equation}
To estimate the gradient term, we use the Jacobi operator $J$ considered
in \cite{kroncke2024foliations}. Although this operator corresponds to a
different notion of stability, its properties are useful here. The proof
of \cite[Proposition~3.10]{kroncke2024foliations} yields, after
integration by parts,
\[
\int_{\Sigma_\sigma}J(\alpha)\alpha\,d\mu
=
\int_{\Sigma_\sigma}
\left(
2|\nabla^h\alpha|^2
-2\mathrm{Sc}^{\Sigma_\sigma}\alpha^2
-(|\vec H|^2-2\mathrm{Sc}^{\Sigma_\sigma})\alpha^2
-E_\sigma\alpha^2
\right)d\mu,
\]
where $E_\sigma$ denotes the collection of the remaining curvature terms
appearing in that formula. By the asymptotic estimates used in the proof
of \cite[Proposition~3.10]{kroncke2024foliations}, these terms satisfy
$\|E_\sigma\|_{L^\infty(\Sigma_\sigma)}=O(\sigma^{-4})$. After
cancellation of the scalar curvature terms, this becomes
\[
\int_{\Sigma_\sigma}2|\nabla^h\alpha|^2\,d\mu
=
\int_{\Sigma_\sigma}J(\alpha)\alpha\,d\mu
+\int_{\Sigma_\sigma}|\vec H|^2\alpha^2\,d\mu
+\int_{\Sigma_\sigma}E_\sigma\alpha^2\,d\mu.
\]
By \cite[Proposition~3.10]{kroncke2024foliations},
\[
\int_{\Sigma_\sigma}J(\alpha)\alpha\,d\mu
\ge
\frac{6m}{\sigma^3}
\int_{\Sigma_\sigma}\alpha^2\,d\mu,
\]
and therefore
\begin{equation}\label{eq:gradient-lower-short}
\int_{\Sigma_\sigma}2|\nabla^h\alpha|^2\,d\mu
\ge
\int_{\Sigma_\sigma}
\left(
|\vec H|^2+\frac{6m}{\sigma^3}-O(\sigma^{-4})
\right)\alpha^2\,d\mu.
\end{equation}
Combining \eqref{eq:gradient-lower-short} with the potential expansion
\eqref{eq:null-potential-asymptotic}, we obtain
\[
\int_{\Sigma_\sigma}
\left(
2|\nabla^h\alpha|^2-V\alpha^2
\right)d\mu
-|\vec H|^2\int_{\Sigma_\sigma}\alpha^2\,d\mu
\ge
\int_{\Sigma_\sigma}
\left(
\mathrm{Sc}^{\Sigma_\sigma}-\frac12|\vec H|^2
+\frac{6m}{\sigma^3}
-O(\sigma^{-4})
\right)\alpha^2\,d\mu.
\]
By \eqref{eq:null-leaf-asymptotics},
$\mathrm{Sc}^{\Sigma_\sigma}-\frac12|\vec H|^2
=\frac{4m}{\sigma^3}+O(\sigma^{-4})$, and hence
\[
\int_{\Sigma_\sigma}
\left(
2|\nabla^h\alpha|^2-V\alpha^2
\right)d\mu
-|\vec H|^2\int_{\Sigma_\sigma}\alpha^2\,d\mu
\ge
\left(
\frac{10m}{\sigma^3}
-O(\sigma^{-4})
\right)
\int_{\Sigma_\sigma}\alpha^2\,d\mu.
\]
Since $m>0$, the positive term $\frac{10m}{\sigma^3}$ dominates the
$O(\sigma^{-4})$ error for all sufficiently large $\sigma$. Thus
\eqref{eq:target-variational} holds, and the leaves are strictly
variationally stable for all sufficiently large $\sigma$.
\end{proof}

\appendix

\section{Auxiliary Rigidity Theorems}\label{appendix}

In this appendix we collect rigidity results used in the equality
cases throughout the paper. These results are classical and are stated
here for the reader's convenience.

\subsection*{Riemannian model geometries}

\subsection*{Euclidean model: Brown-York rigidity}

In the time-symmetric case $K=0$, the Kijowski-Liu-Yau energy
reduces to the Brown-York mass, and rigidity follows from
the fundamental result of Shi-Tam.  
\begin{theorem}[Shi-Tam {\cite[Theorem 1]{shi2002positive}}]\label{rigiditybrown}
 Let \((\Omega, g)\) be a compact orientable manifold of dimension three with a smooth boundary and with nonnegative scalar curvature. Suppose \(\partial \Omega\) has finitely many components \(\Sigma_i\) such that each component has positive Gaussian curvature and positive mean curvature \(H^i \) with respect to the unit outward normal. 
Then for each component,
\begin{equation}\label{inequbrown}
\int_{\Sigma_i} H^i\, d\mu
\le
\int_{\Sigma_i} H_0^i\, d\mu,
\end{equation}
where $H_0^i$ denotes the mean curvature of the unique convex
isometric embedding of $\Sigma_i$ into $\mathbb{R}^3$.
Moreover, if equality holds for some $\Sigma_i$,
then $\partial\Omega$ is connected and $\Omega$
is isometric to a domain in $\mathbb{R}^3$.
\end{theorem}
The existence and uniqueness (up to rigid motions) of the convex
isometric embedding into $\mathbb{R}^3$ follow from the classical
Weyl-Nirenberg-Pogorelov theorem.
\begin{theorem}[Weyl-Nirenberg-Pogorelov]\label{WeyNi}
Let $(S^2,g)$ be a $C^{k,\alpha}$ ($k\ge3$, $\alpha\in(0,1)$)
Riemannian $2$-sphere with Gaussian curvature $K_g>0$.
Then there exists a strictly convex isometric embedding
\[
X\colon (S^2,g)\hookrightarrow(\mathbb R^3,g_{\mathrm{Eucl}}),
\]
unique up to orientation-preserving rigid motions of $\mathbb R^3$.
\end{theorem}
In \cite{shi2002positive} Shi and Tam also proved a higher dimensional version of Theorem \ref{rigiditybrown}: 
\begin{theorem}[{\cite[Theorem 4.1]{shi2002positive}}]\label{highershitam}     
     Let \((\Omega, g)\)  be a compact orientable Riemannian manifold of dimension $n \geq 3$, with smooth boundary  $\partial \Omega$ and nonnegative scalar curvature. Assume \(3 \leq n \leq 7\) or \(\Omega\) is spin. Suppose the boundary has finitely many connected components \(\Sigma_i\) such that each component has positive mean curvature \(H^i \) with respect to the unit outward normal and can be isometrically embedded in \(\mathbb{R}^n\) as a convex hypersurface. Then for each component
\[
\int_{\Sigma_i} H^i\, d\mu
\le
\int_{\Sigma_i} H_0^i\, d\mu,
\]
where \(H^i_0\) is the mean curvature of the isometric embedding of \(\Sigma_i\) in the Euclidean space. Moreover, if equality holds for some $\Sigma_i$, then $\partial\Omega$ is connected (i.e., $\partial\Omega = \Sigma_i$) and $\Omega$ is isometric to a domain in $\mathbb{R}^n$. 
\end{theorem}
\begin{remark}\label{remarkhigherdim}  
The dimensional and spin restrictions in Theorem \ref{highershitam} arise from the use of the positive mass theorem in the original proof. Extensions of the positive mass theorem to higher dimensions without the spin assumption have been announced by Lohkamp
\cite{lohkamp2016higher1, lohkamp2016higher2} and by Schoen-Yau
\cite{schoen2017positive}. More recently, new proofs of the positive mass theorem have been announced by Bi et al.~\cite{bi2026proof} for dimensions up to $19$, and by Brendle and Wang \cite{brendle2026dimension} for arbitrary dimensions. Whenever Theorem \ref{highershitam} is applied in dimensions $n \ge 8$ in this paper, the corresponding statements should therefore be understood under the assumption that the positive mass theorem holds in that dimension.
\end{remark}

To identify the geometry of these higher-dimensional embeddings when the intrinsic scalar curvature is constant, we rely on the following classical rigidity theorem by Ros.

\begin{theorem}[Ros {\cite[Theorem 1]{ros1988compact}}]\label{rosrigidity}
Let $\Sigma^{n-1} \subset \mathbb{R}^n$ be a closed, connected, embedded hypersurface with constant intrinsic scalar curvature. Then $\Sigma$ is a round sphere.
\end{theorem}

\subsection*{Hyperbolic model}

For intrinsically round boundary spheres, the
hyperbolic rigidity theorem of Shi and Tam
takes the following form.

\begin{theorem}[Shi-Tam 
{\cite[Theorem 3.8]{shi2007rigidity}}]
\label{rigidityBYhyper}
Let $(\Omega^3,g)$ be a compact, connected,
orientable Riemannian manifold with smooth
boundary $\Sigma$. Suppose that $
\operatorname{Sc}^{\Omega}\geq 2\Lambda$, $ \Lambda<0$,
and that $(\Sigma,h)$ is intrinsically isometric
to a round sphere of Gauss curvature $K_\Sigma>0$,
with positive outward mean curvature $H$. Then $\Sigma$ admits a convex isometric embedding
into $\mathbb H^3_{\Lambda/3}$, unique up to
ambient isometries. Since $(\Sigma,h)$ is
intrinsically round, its image is a geodesic
sphere. Let $H_0$ denote the mean curvature
of this embedding with respect to the
outward unit normal.
Then
\[
\int_\Sigma (H_0-H)\,d\mu\geq 0.
\]
Equality holds if and only if $(\Omega,g)$
is isometric to a geodesic ball in
$\mathbb H^3_{\Lambda/3}$.
\end{theorem}
\subsection*{Spherical model}

In the positive curvature setting, rigidity requires stronger
curvature assumptions. The following results of Hang-Wang provide the appropriate model rigidity.

\begin{theorem}[Hang-Wang {\cite[Theorem 2]{hang2009rigidity}}]\label{hangwang}
Let $(M,g)$ be a compact $n$-dimensional Riemannian manifold ($n\ge2$)
with nonempty boundary $\Sigma$.
Assume $\mathrm{Ric}^M \ge (n-1)g$,
$(\Sigma,g_\Sigma)$ is isometric to  the
standard unit sphere $\mathbb S^{n-1}$,
and the second fundamental form of $\Sigma$ is nonnegative.
Then $(M,g)$ is isometric to the hemisphere $\mathbb S^n_+$.
\end{theorem}

\begin{theorem}[Hang-Wang {\cite[Theorem 3]{hang2009rigidity}}]\label{hangwang3}
Let $(M,g)$ be a smooth compact $n$-dimensional Riemannian manifold with boundary $\partial M = \Sigma$, and let $\Omega \subset \mathbb{S}^n_+$ be a compact domain with smooth boundary in the open hemisphere. Suppose:
\begin{itemize}
    \item $\mathrm{Ric}^M \ge (n-1)g$,
    \item there is an isometric embedding $\iota \colon (\Sigma, g_\Sigma) \to \partial\Omega$,
    \item $B \ge \iota^*B_0$, where $B$ is the second fundamental form of $\Sigma$ in $M$ and $B_0$ is the second fundamental form of $\partial\Omega$ in $\mathbb{S}^n_+$.
\end{itemize}
Then $(M,g)$ is isometric to $(\Omega, g_{\mathbb{S}^n_+})$.
\end{theorem}
Rigidity results of this type can be viewed as Ricci-strengthened solutions to boundary rigidity problems on the sphere, such as Min-Oo's conjecture \cite{MinOo1998}. The original scalar-curvature formulation of that conjecture was disproved by Brendle, Marques, and Neves \cite{brendle2011deformations}, highlighting the necessity of the stronger Ricci curvature bound in this positively curved setting.

\subsection*{Spacetime quasi-local rigidity}

For the Lorentzian rigidity arguments in Section~\ref{sectionSTCMC}, we also use
the following positivity and rigidity results for the
Kijowski-Liu-Yau quasi-local energy.
\begin{theorem}[Liu-Yau {\cite[Theorem 1]{liu2003positivity,liu2006positivity}}]
\label{liuyaurigi}
Let $(\Omega,g,K)$ be a compact initial data set satisfying
the dominant energy condition. Suppose $\partial\Omega$ has finitely many components $\Sigma_i$,
each mean-convex, with positive Gaussian curvature and spacelike mean curvature vector.
Then
\[
\mathcal{E}_{KLY}(\Sigma_i)\ge 0.
\]
Moreover, if equality holds for some component,
then $\partial\Omega$ is connected and
$\Omega$ is isometric to a spacelike hypersurface in Minkowski spacetime.  Specifically, $\Omega$ can be isometrically embedded in $\mathbb{R}^{3,1}$ as a spacelike graph $(x, f(x))$ over a spatial domain $\Omega_0 \subset \mathbb{R}^3$, where $f$ is a smooth function on $\Omega_0$ that vanishes on $\partial\Omega_0$.
\end{theorem}

The rigidity of the Kijowski-Liu-Yau energy in Minkowski spacetime is even stronger; this was first observed by Ó Murchadha and Szabados in \cite{murchadha2004comment} and was later fully characterized by Miao, Shi, and Tam in the following result. 
\begin{theorem}[{\cite[Theorem 4.1]{miao2010geometric}}]\label{minkowskirigi}
    Let $\Sigma$ be a closed, connected, smooth, spacelike $2$-surface in Minkowski spacetime $\mathbb{R}^{3,1}$. 
    Suppose $\Sigma$ spans a compact spacelike hypersurface in $\mathbb{R}^{3,1}$.
    If $\Sigma$ has positive Gaussian curvature and a spacelike mean curvature vector, then $
    \mathcal{E}_{KLY} (\Sigma) \geq 0$.
    Moreover, $$\mathcal{E}_{KLY} (\Sigma) = 0$$ if and only if $\Sigma$ lies on a hyperplane in $\mathbb{R}^{3,1}$.
\end{theorem}

\subsection*{Spectral estimates on the sphere}

For the alternative rigidity argument in Section~\ref{sectionSTCMC}, we
also use the following estimate of El Soufi and Ilias for the second
eigenvalue of a Schrödinger operator on a genus-zero surface.

\begin{theorem}[El Soufi-Ilias {\cite{el1992majoration}}]
\label{thm:ElSoufiIlias}
Let $(\Sigma,h)$ be a closed surface of genus zero, let
$q\in C^0(\Sigma)$, and define the second eigenvalue $
\lambda_2(-\Delta_h+q)$. Then
\[
\lambda_2(-\Delta_h+q)\,|\Sigma|
\le
8\pi+\int_\Sigma q\,d\mu.
\]
Moreover, equality holds if and only if $\Sigma$ admits a conformal map
into the standard $\mathbb S^2$ whose coordinate functions are second
eigenfunctions of $-\Delta_h+q$.
\end{theorem}

\vspace{0.5 cm}

 \paragraph*{\emph{Acknowledgements.}} 
The author thanks Carla Cederbaum for helpful discussions in Stockholm and for encouraging him to revisit the stability of small STCMC spheres.

\bibliographystyle{amsplain}
\bibliography{Lit_new}

@article {Miao,
    AUTHOR = {Miao, Pengzi and Wang, Yaohua and Xie, Naqing},
     TITLE = {On {H}awking mass and {B}artnik mass of {CMC} surfaces},
   JOURNAL = {Math. Res. Lett.},
  FJOURNAL = {Mathematical Research Letters},
    VOLUME = {27},
      YEAR = {2020},
    NUMBER = {3},
     PAGES = {855--885},
      ISSN = {1073-2780},
   MRCLASS = {53C20 (53C42)},
  MRNUMBER = {4216572},
MRREVIEWER = {Rafael L\'{o}pez},
       DOI = {10.4310/MRL.2020.v27.n3.a12},
       URL = {https://doi.org/10.4310/MRL.2020.v27.n3.a12},
}

@incollection {Chriyau,
    AUTHOR = {Christodoulou, Demetrios and Yau, Shing-Tung},
     TITLE = {Some remarks on the quasi-local mass},
 BOOKTITLE = {Mathematics and general relativity ({S}anta {C}ruz, {CA},
              1986)},
    SERIES = {Contemp. Math.},
    VOLUME = {71},
     PAGES = {9--14},
 PUBLISHER = {Amer. Math. Soc., Providence, RI},
      YEAR = {1988},
   MRCLASS = {83C99 (58E12)},
  MRNUMBER = {954405},
MRREVIEWER = {K. P. Tod},
       DOI = {10.1090/conm/071/954405},
       URL = {https://doi.org/10.1090/conm/071/954405},
}

@article {Ye,
    AUTHOR = {Ye, Rugang},
     TITLE = {Foliation by constant mean curvature spheres},
   JOURNAL = {Pacific J. Math.},
  FJOURNAL = {Pacific Journal of Mathematics},
    VOLUME = {147},
      YEAR = {1991},
    NUMBER = {2},
     PAGES = {381--396},
      ISSN = {0030-8730},
   MRCLASS = {53C12},
  MRNUMBER = {1084717},
       URL = {http://projecteuclid.org/euclid.pjm/1102644918},
}

@article {Living,
    AUTHOR = {Szabados,   László Benő},
     TITLE = {Quasi-Local Energy-Momentum and Angular Momentum in GR},
   JOURNAL = {Living Rev. Relativity},
    VOLUME = {7},
      YEAR = {2004},
    NUMBER = {4},
}

@article {STCMC,
    AUTHOR = {Cederbaum, Carla and Sakovich, Anna},
     TITLE = {On center of mass and foliations by constant spacetime mean
              curvature surfaces for isolated systems in {G}eneral
              {R}elativity},
   JOURNAL = {Calc. Var. Partial Differential Equations},
  FJOURNAL = {Calculus of Variations and Partial Differential Equations},
    VOLUME = {60},
      YEAR = {2021},
    NUMBER = {6},
     PAGES = {Paper No. 214},
      ISSN = {0944-2669},
   MRCLASS = {53C21 (58J37 83C05)},
  MRNUMBER = {4305436},
       DOI = {10.1007/s00526-021-02060-z},
       URL = {https://doi.org/10.1007/s00526-021-02060-z},
}

@article {HY,
    AUTHOR = {Huisken, Gerhard and Yau, Shing-Tung},
     TITLE = {Definition of center of mass for isolated physical systems and
              unique foliations by stable spheres with constant mean
              curvature},
   JOURNAL = {Invent. Math.},
  FJOURNAL = {Inventiones Mathematicae},
    VOLUME = {124},
      YEAR = {1996},
    NUMBER = {1-3},
     PAGES = {281--311},
      ISSN = {0020-9910},
   MRCLASS = {53C20 (83C99)},
  MRNUMBER = {1369419},
MRREVIEWER = {Alan D. Rendall},
       DOI = {10.1007/s002220050054},
       URL = {https://doi.org/10.1007/s002220050054},
}

@article {Hawma,
    AUTHOR = {Hawking, Stephen W.},
     TITLE = {Gravitational radiation in an expanding universe},
   JOURNAL = {J. Mathematical Phys.},
  FJOURNAL = {Journal of Mathematical Physics},
    VOLUME = {9},
      YEAR = {1968},
    NUMBER = {4},
     PAGES = {598--604},
      ISSN = {0022-2488},
   MRCLASS = {83C30},
  MRNUMBER = {3960907},
       DOI = {10.1063/1.1664615},
       URL = {https://doi.org/10.1063/1.1664615},
}

@article{main1local,
title = {Local space time constant mean curvature and constant expansion foliations},
journal = {Journal of Geometry and Physics},
volume = {188},
pages = {104823},
year = {2023},
issn = {0393-0440},
doi = {https://doi.org/10.1016/j.geomphys.2023.104823},
url = {https://www.sciencedirect.com/science/article/pii/S039304402300075X},
author = {Metzger, Jan and Peñuela Diaz, Alejandro },
}

@article{sun2017rigidity,
  title={Rigidity of Hawking mass for surfaces in three manifolds},
  author={Sun, Jiacheng},
  journal={Pacific Journal of Mathematics},
  volume={292},
  number={1},
  pages={257--282},
  year={2017},
  publisher={Mathematical Sciences Publishers}
}

@article{shi2019uniqueness,
  title={Uniqueness of the mean field equation and rigidity of Hawking mass},
  author={Shi, Yuguang and Sun, Jiacheng and Tian, Gang and Wei, Dongyi},
  journal={Calculus of Variations and Partial Differential Equations},
  volume={58},
  pages={1--16},
  year={2019},
  publisher={Springer}
}

@article{shi2002positive,
  title={Positive mass theorem and the boundary behaviors of compact manifolds with nonnegative scalar curvature},
  author={Shi, Yuguang and Tam, Luen-Fai},
  journal={Journal of Differential Geometry},
  volume={62},
  number={1},
  pages={79--125},
  year={2002},
  publisher={Lehigh University}
}

@article{shi2007rigidity,
  title={Rigidity of compact manifolds and positivity of quasi-local mass},
  author={Shi, Yuguang and Tam, Luen-Fai},
  journal={Classical and Quantum Gravity},
  volume={24},
  number={9},
  pages={2357},
  year={2007},
  publisher={IOP Publishing}
}

@article{liu2003positivity,
  title={Positivity of quasilocal mass},
  author={Liu, Chiu-Chu Melissa and Yau, Shing-Tung},
  journal={Physical review letters},
  volume={90},
  number={23},
  pages={231102},
  year={2003},
  publisher={APS}
}

@article{liu2006positivity,
  title={Positivity of quasi-local mass II},
  author={Liu, Chiu-Chu Melissa and Yau, Shing-Tung},
  journal={Journal of the American Mathematical Society},
  volume={19},
  number={1},
  pages={181--204},
  year={2006}
}

@article{schoen2017positive,
author = {Schoen, Richard and Yau, Shing-Tung},
year = {2019},
month = {04},
pages = {441-480},
title = {Positive Scalar Curvature and Minimal Hypersurface Singularities},
volume = {24},
journal = {Surveys in Differential Geometry},
doi = {10.4310/SDG.2019.v24.n1.a10}
}

@article{lohkamp2016higher2,
  title={The higher dimensional positive mass theorem II},
  author={Lohkamp, Joachim},
  journal={arXiv preprint arXiv:1612.07505},
  year={2016}
}

@article{lohkamp2016higher1,
  title={The higher dimensional positive mass theorem I},
  author={Lohkamp, Joachim},
  journal={arXiv preprint arXiv:math/0608795},
  year={2016}
}

@article{ros1988compact,
  title={Compact hypersurfaces with constant scalar curvature and a congruence theorem},
  author={Ros, Antonio},
  journal={Journal of Differential Geometry},
  volume={27},
  number={2},
  pages={215--220},
  year={1988},
  publisher={Lehigh University}
}

@article{miao2010geometric,
  title={On geometric problems related to Brown-York and Liu-Yau quasilocal mass},
  author={Miao, Pengzi and Shi, Yuguang and Tam, Luen-Fai},
  journal={Communications in mathematical physics},
  volume={298},
  number={2},
  pages={437--459},
  year={2010},
  publisher={Springer}
}

@article{murchadha2004comment,
  title={Comment on “Positivity of quasilocal mass”},
  author={Ó Murchadha,  Niall and Szabados, László Benő and Tod, Paul},
  journal={Physical review letters},
  volume={92},
  number={25},
  pages={259001},
  year={2004},
  publisher={APS}
}

@article{melo2024hawking,
  title={On the Hawking mass for CMC surfaces in positive curved 3-manifolds},
  author={Melo, Luiz Ricardo},
  journal={Proceedings of the American Mathematical Society},
  volume={152},
  number={12},
  pages={5373--5380},
  year={2024}
}

@article{hang2009rigidity,
  title={Rigidity theorems for compact manifolds with boundary and positive Ricci curvature},
  author={Hang, Fengbo and Wang, Xiaodong},
  journal={Journal of Geometric Analysis},
  volume={19},
  pages={628--642},
  year={2009},
  publisher={Springer}
}

@incollection{MinOo1998,
  author    = {Min-Oo, Maung},
  title     = {Scalar Curvature Rigidity of Certain
               Symmetric Spaces},
  booktitle = {Geometry, Topology, and Dynamics},
  editor    = {Lalonde, Fran{\c{c}}ois},
  series    = {CRM Proceedings \& Lecture Notes},
  volume    = {15},
  pages     = {127--136},
  publisher = {American Mathematical Society},
  address   = {Providence, RI},
  year      = {1998},
  note      = {Proceedings of the CRM Workshop,
               Montr{\'e}al, 1995}
}

@article{brendle2011deformations,
  title={Deformations of the hemisphere that increase scalar curvature},
  author={Brendle, Simon and Marques, Fernando C and Neves, Andre},
  journal={Inventiones mathematicae},
  volume={185},
  number={1},
  pages={175--197},
  year={2011},
  publisher={Springer}
}

@article{mars2012stability,
  title={Stability of MOTS in totally geodesic null horizons},
  author={Mars, Marc},
  journal={Classical and Quantum Gravity},
  volume={29},
  number={14},
  pages={145019},
  year={2012},
  publisher={IOP Publishing}
}

@article{Hersch1970,
  author  = {Hersch, Joseph},
  title   = {Quatre propriétés isopérimétriques de membranes sphériques homogènes},
  journal = {C. R. Acad. Sci. Paris Sér. A-B},
  volume  = {270},
  year    = {1970},
  pages   = {A1645--A1648}
}

@article{choquet2009light,
  title={The light-cone theorem},
  author={Choquet-Bruhat, Yvonne and Chru{\'s}ciel, Piotr T and Mart{\'\i}n-Garc{\'\i}a, Jos{\'e} M},
  journal={Classical and Quantum Gravity},
  volume={26},
  number={13},
  pages={135011},
  year={2009},
  publisher={IOP Publishing}
}

@article{el1992majoration,
  title={Majoration de la seconde valeur propre d'un op{\'e}rateur de Schr{\"o}dinger sur une vari{\'e}t{\'e} compacte et applications},
  author={El Soufi, Ahmad and Ilias, Sa{\i}d},
  journal={Journal of functional analysis},
  volume={103},
  number={2},
  pages={294--316},
  year={1992},
  publisher={Elsevier}
}

@article{diaz2025rigidity,
  title={Rigidity and positivity of Hawking quasi-local energy on area-constrained critical surfaces},
  author={Peñuela Diaz, Alejandro },
  journal={arXiv preprint arXiv:2507.16588},
  year={2025}
}

@article{chen1973surface,
  title={On the surface with parallel mean curvature vector},
  author={Chen, Bang-yen},
  journal={Indiana University Mathematics Journal},
  volume={22},
  number={7},
  pages={655--666},
  year={1973},
  publisher={JSTOR}
}

@article{yau1974submanifolds,
  title={Submanifolds with constant mean curvature},
  author={Yau, Shing-Tung},
  journal={American Journal of Mathematics},
  volume={96},
  number={2},
  pages={346--366},
  year={1974},
  publisher={JSTOR}
}

@article{chen2013submanifolds,
  title={Submanifolds with parallel mean curvature vector in Riemannian and indefinite space forms},
  author={Chen, Bang-Yen},
  journal={Arab Journal of Mathematical Sciences},
 volume={16},
  number={1},
  pages={1--46},
  year={2010},
}

@article{brendlerigidity,
  title={Rigidity phenomena involving scalar curvature},
  author={Brendle, Simon},
  journal={Surveys in Differential Geometry},
  volume={17},
  number={1},
  pages={179--202},
  year={2012},
  publisher={International Press of Boston}
}

@article{andersson2010curvature,
  title={Curvature estimates for stable marginally trapped surfaces},
  author={Andersson, Lars and Metzger, Jan},
  journal={Journal of differential geometry},
  volume={84},
  number={2},
  pages={231--265},
  year={2010},
  publisher={Lehigh University}
}

@article{alaee2021stable,
  title={Stable surfaces and free boundary marginally outer trapped surfaces},
  author={Alaee, Aghil and Lesourd, Martin and Yau, Shing-Tung},
  journal={Calculus of Variations and Partial Differential Equations},
  volume={60},
  number={5},
  pages={186},
  year={2021},
  publisher={Springer}
}

@article{wolff2024lellis,
  title={A De Lellis--M{\"u}ller type estimate on the Minkowski lightcone},
  author={Wolff, Markus},
  journal={Calculus of Variations and Partial Differential Equations},
  volume={63},
  number={7},
  pages={185},
  year={2024},
  publisher={Springer}
}

@article{wolff2024effects,
  title={On effects of the null energy condition on totally umbilic hypersurfaces in a class of static spacetimes},
  author={Wolff, Markus},
  journal={Annals of Global Analysis and Geometry},
  volume={66},
  number={3},
  pages={10},
  year={2024},
  publisher={Springer}
}

@article {kroncke2024foliations,
    AUTHOR = {Kr\"oncke, Klaus and Wolff, Markus},
     TITLE = {Foliations of asymptotically {S}chwarzschildean lightcones by
              surfaces of constant spacetime mean curvature},
   JOURNAL = {Math. Ann.},
  FJOURNAL = {Mathematische Annalen},
    VOLUME = {394},
      YEAR = {2026},
    NUMBER = {3},
     PAGES = {73},
      ISSN = {0025-5831,1432-1807},
   MRCLASS = {53C21 (53C50)},
  MRNUMBER = {5036070},
       DOI = {10.1007/s00208-026-03331-w},
       URL = {https://doi.org/10.1007/s00208-026-03331-w},
}

@article{tenan2025volume,
  title={Volume preserving spacetime mean curvature flow and foliations of initial data sets},
  author={Tenan, Jacopo},
  journal={Journal of Functional Analysis},
  pages={111313},
  year={2025},
  publisher={Elsevier}
}

@book{choquet2009general,
  title={General relativity and the Einstein equations},
  author={Choquet-Bruhat, Yvonne},
  year={2009},
  publisher={Oxford university press}
}

@article{bi2026proof,
  title={A proof for the Riemannian positive mass theorem up to dimension 19},
  author={Bi, Yuchen and Hao, Tianze and He, Shihang and Shi, Yuguang and Zhu, Jintian},
  journal={arXiv preprint arXiv:2603.02769},
  year={2026}
}

@article{LiYau1982,
  title={A new conformal invariant and its applications to the Willmore conjecture and the first eigenvalue of compact surfaces},
  author={Li, Peter and Yau, Shing-Tung},
  journal={Inventiones mathematicae},
  volume={69},
  number={2},
  pages={269--291},
  year={1982},
  publisher={Springer-Verlag Berlin/Heidelberg}
}

@article{eschenburg1988constant,
  title={Constant mean curvature surfaces in $4 $-space forms},
  author={Eschenburg, Jost-Hinrich and Tribuzy, Renato de Azevedo},
  journal={Rendiconti del Seminario matematico della Universit{\`a} di Padova},
  volume={79},
  pages={185--202},
  year={1988}
}

@article{brendle2026dimension,
  title={A dimension descent scheme for the positive mass theorem in arbitrary dimension},
  author={Brendle, Simon and Wang, Yipeng},
  journal={arXiv preprint arXiv:2604.08473},
  year={2026}
}
\end{document}